\documentclass{amsart}
\usepackage{bbm}
\usepackage{bbding}
\usepackage[mathcal]{euscript}
\usepackage{graphicx}
\usepackage{amsthm}
\usepackage{calc}
\usepackage{mathrsfs,dsfont}
\usepackage{CJK,fancyhdr,amscd}
\usepackage{anysize} 
\usepackage{amsmath,amssymb,amsfonts}
\usepackage{epsfig,enumerate}
\usepackage{latexsym}
\usepackage{indentfirst, latexsym}
\usepackage{graphics}
\usepackage[all,poly,knot]{xy}
\usepackage[colorlinks,linkcolor=blue,anchorcolor=blue,citecolor=blue,pagebackref]{hyperref}
\usepackage{geometry}
\numberwithin{equation}{section}
\newtheorem{theorem} {Theorem} [section]
\newtheorem{proposition}[theorem]{Proposition}

\newtheorem{lemma}  [theorem]     {Lemma}
\newtheorem{question}  [theorem]     {Question}

\newtheorem{remark}  [theorem]     {Remark}

\newtheorem{conjecture}  [theorem]     {Conjecture}

\theoremstyle{definition}
\newtheorem{definition}  [theorem]     {Definition}

\newtheorem{example}  [theorem]     {Example}

\renewcommand{\a}{\alpha}
\renewcommand{\b}{\beta}
\newcommand{\g}{\gamma}
\renewcommand{\o}{\omega}
\newcommand{\eps}{\epsilon}

\renewcommand{\gg}{\mathfrak g}
\newcommand{\gk}{\mathfrak k}
\newcommand{\ga}{\mathfrak a}
\newcommand{\gq}{\mathfrak q}
\newcommand{\gz}{\mathfrak z}
\newcommand{\gp}{\mathfrak p}
\newcommand{\gl}{\mathfrak l}
\newcommand{\gh}{\mathfrak h}
\newcommand{\gm}{\mathfrak m}
\newcommand{\gu}{\mathfrak u}
\newcommand{\gt}{\mathfrak t}
\newcommand{\ad}{\rm {ad}}
\newcommand{\SU}{\rm {SU}}
\newcommand{\so}{\mathfrak{so}}
\newcommand{\su}{\mathfrak{su}}
\newcommand{\T}{\rm{T}}

\newcommand{\ep}{\epsilon}

\newcommand{\Lb}{\Lambda^b}

\renewcommand*\backref[1]{}
\renewcommand*\backrefalt[4]{ \ifcase #1 \or (cited on page #2) \else (cited on pages #2) \fi}

\begin{document}

\title{A note on compact homogeneous manifolds with Bismut parallel torsion}

\begin{abstract}
In this article, we investigate the class of Hermitian manifolds whose Bismut connection has parallel torsion ({\rm BTP} for brevity). In particular, we focus on the case where the manifold is (locally) homogeneous with respect to a group of holomorphic isometries and we fully characterize the compact Chern flat {\rm BTP} manifolds. Moreover we show that certain compact flag manifolds are {\rm BTP} if and only if the metric is K\"ahler or induced by the Cartan-Killing form and we then characterize {\rm BTP} invariant metrics on compact semisimple Lie groups which are Hermitian w.r.t. a Samelson structure and are projectable along the Tits fibration. We state a conjecture concerning the question when the Bismut connection of a  BTP compact Hermitian locally homogeneous manifold has parallel curvature, giving examples and providing evidence in some special cases.
\end{abstract}

\author{Fabio Podest\`a}
\address{Fabio Podest\`a. Dipartimento di Matematica e Informatica ``U. Dini", Universit\`a degli Studi di Firenze,
Viale Morgagni 67/a, 50134 Firenze, Italy.} \email{fabio.podesta@unifi.it}
\thanks{Podest\`a was supported by GNSAGA of INdAM and by the project PRIN 2022AP8HZ9 ``Real and Complex
Manifolds: Geometry and Holomorphic Dynamics". Zheng is partially supported by NSFC grant 12071050 and 12141101, a Chongqing grant cstc2021ycjh-bgzxm0139, and the 111 Project D21024.}

\author{Fangyang Zheng}
\address{Fangyang Zheng. School of Mathematical Sciences, Chongqing Normal University, Chongqing 401331, China}
\email{20190045@cqnu.edu.cn;\ franciszheng@yahoo.com} \thanks{}



\maketitle

\tableofcontents

\section{Introduction}
Given a Riemannian manifold $(M,g)$, the study of metric connections with skew-symmetric torsion has gained a remarkable interest, as they come up in different meaningful geometric structures and are a powerful tool in Type II string theory and in supergravity theories (see e.g.\,\cite{Ag}). In particular metric connections whose torsion is skew-symmetric and parallel have many interesting properties, e.g. their curvature tensor is a symmetric endomorphism of $\Lambda^2(TM)$, and have become an important research topic in the recent years. Basic examples are given by the canonical connections of naturally reductive homogeneous spaces, Sasakian and $3$-Sasakian spaces, nearly K\"ahler manifolds and nearly parallel $G_2$-structures in dimension $7$ among others (see e.g.). Some classification results for manifolds carrying connections with parallel skew-symmetric torsion have been recently achieved considering the holonomy action on the tangent space (see \cite{CS}, \cite{CMS}).

In the Hermitian setting, it is well-known that a Hermitian manifold $(M^n,g,J)$ admits a unique metric connection $\nabla^b$ which leaves $J$ parallel and whose torsion $T^b$ is skew-symmetric. The connection $\nabla^b$ is called the {\it Bismut connection} (see \cite{Bismut}), although previously introduced by Strominger (\cite{Strominger}), and its torsion $3$-form coincides with $-d^c\omega$, where $\omega=g(J\cdot,\cdot)$ is the corresponding fundamental $2$-form. Hermitian manifolds endowed with the Bismut connection are called {\it K\"ahler with torsion (KT)} and when the torsion form is also closed or, equivalently $\partial\overline\partial \omega=0$, the KT-structure is called a {\it strong KT structure (SKT)} and the Hermitian metric $g$ is said pluriclosed.  In recent years the Bismut connection has turned out to be a very meaningful object in different research topics, including Hermitian curvature flows (see e.g.\,\cite{ST}) as well as the search for special Hermitian connection with remarkable geometric and algebraic properties (see e.g. \cite{AV},\cite{AOUV},\cite{FT}).

In this paper we will focus on the case when the Bismut connection of a Hermitian manifold $(M,g,J)$ is supposed to be $\nabla^b$-parallel. In this case, according to \cite{ZhaoZ22}, we will call the manifold (and sometimes also the metric) {\it Bismut torsion parallel} ({\rm BTP} for brevity). In \cite{ZhaoZ19Str} it has been proved amongst others that for {\rm BTP} manifolds the pluriclosed condition is equivalent to saying that the manifold satisfies the Bismut K\"ahler-like condition, while in \cite{ZhaoZ22} several general properties of {\rm BTP} manifolds are proved, showing e.g. that on a compact {\rm BTP} manifold the metric $g$ is Gauduchon while the holonomy of $\nabla^b$ reduces to a subgroup of $U(n-1)$ if $g$ is not balanced. There are a number of papers which contributed to this observation. As a partial list see  \cite{AV} and \cite{ZhaoZ22} for example.

An important class of non-K\"ahler manifolds are the Vaisman manifolds, namely those compact non-K\"ahler, locally conformally K\"ahler manifolds whose Lee form is parallel w.r.t. the Levi-Civita connection. As proved in \cite{AV}, any Vaisman manifold is BTP.  Another subset of BTP is the so-called Bismut K\"ahler-like (or BKL for brevity) manifolds, which are Hermitian manifolds whose Bismut curvature tensors obeys all K\"ahler symmetries. By the main result of \cite{ZhaoZ19Str}, any  BKL metric is BTP.

When $n=2$, Vaisman, BKL, and BTP conditions are all equivalent, and compact non-K\"ahler Vaismann surfaces are fully classified by Belgun in \cite{B00}. They are non-K\"ahler properly elliptic surfaces, Kodaira surfaces, and (Class 1 and elliptic) Hopf surfaces. When $n\geq 3$, however, Vaisman and BKL become disjoint sets: a result proved in \cite{YZZ} says that for any compact non-K\"ahler BKL manifold $(M^n,g)$ with $n\geq 3$, $M^n$ does not admit any Vaisman metric.

For $n\geq 3$, there are other non-balanced BTP manifolds which are not BKL or Vaisman. There are also balanced (non-K\"ahler) BTP manifolds. We refer the readers to \cite{ZhaoZ19Nil} and \cite{ZhaoZ22} for more examples of BTP manifolds. They form a rich yet highly restrictive class of special Hermitian manifolds.

In this paper, we will be  concerned with BTP manifolds that are (locally) homogeneous Hermitian manifolds. By Ambrose-Singer Theorem (see e.g.\,\cite{AS},\cite{TV}) and its Hermitian extension due to Sekigawa (\cite{Sekigawa}), it is well known that a complete simply-connected
Hermitian manifold $(M,g,J)$ is homogeneous if and only if there exists a Hermitian connection $\nabla$, called an {\it Ambrose-Singer connection}, whose torsion and curvature are $\nabla$-parallel. In particular if there exists a Hermitian Ambrose-Singer connection whose torsion is skew-symmetric, the manifold $(M,g,J)$ turns out to be a naturally reductive homogeneous space with respect to the action of a transitive group of biholomorphic isometries.

 Clearly, a Hermitian Ambrose-Singer connection with skew-symmetric torsion on a Hermitian manifold $(M,g,J)$ coincides with the Bismut connection of the metric $g$. Following \cite{NiZheng} we give the following
\begin{definition}\label{defAS} A {\rm BTP} Hermitian manifold $(M,g,J)$ is called {\em Bismut Ambrose-Singer} ({\rm BAS} in short) if its Bismut connection has parallel torsion and curvature or, equivalently, if its Bismut connection is an Ambrose-Singer connection.
\end{definition}

For instance, every compact flag manifold $F=K/H$ where $K$ is a compact semisimple Lie group and $H\subset K$ is the centralizer of a torus in $K$, can be endowed with many invariant complex structures $J$ and the metric $g_o$ on $F$ induced by the Cartan-Killing form of the Lie algebra $\gk$ of $K$ is automatically Hermitian w.r.t. $J$, with  its Bismut connection being the associated canonical connection, hence Ambrose-Singer.

The aim of this paper is to investigate the geometry of (locally) homogeneous Hermitian manifolds whose Bismut connection is BTP and in particular to study conditions under which the Bismut connection is Ambrose-Singer.

Our first main result concerns the case of compact Chern flat Hermitian manifolds $(M,g,J)$, which by a well-known result due to Boothby (\cite{Boothby}) are covered by a complex Lie group endowed with a left-invariant Hermitian metric. We first prove the following characterization
\begin{theorem}\label{ThmChernflat} Let $(M^n,g)$ be a compact Chern flat {\rm BTP} manifold. Then its universal cover is a complex reductive Lie group $G$ equipped with a left-invariant metric, where $G={\mathbb C}^k\times G_1\times \cdots \times G_r$ is the orthogonal direct product with  each $G_i$  a simple complex Lie group.

Viceversa let $G={\mathbb C}^k \times G_1\times \cdots \times G_r$ where each $G_i$ is any connected, simply-connected simple complex Lie group. Then there exists a left-invariant Hermitian metric $g$ on $G$ which is {\rm BTP}.\end{theorem}

In Proposition \ref{ChernflatBAS} we prove that any {\rm BTP} manifold as in Theorem \ref{ThmChernflat} is actually {\rm BAS} while in Theorem \ref{thm2.6} we address the problem of describing the moduli space of such {\rm BTP}  metrics.

We then consider the case of a simply-connected generalized flag manifold $F$, namely a homogeneous space $K/H$ where $K$ is a compact semisimple Lie group and $H$ is the centralizer in $K$ of a torus. It is well-known that this class of manifolds exhausts the set of all compact simply-connected K\"ahler homogeneous spaces and we investigate the properties of invariant BTP metrics on such manifolds. After recalling the basic facts concerning the existence of invariant complex structures $J$ on a flag manifold $F$, we consider $J$-Hermitian invariant metrics showing that they are automatically balanced and we then prove a characterization of BTP metrics on a special remarkable subclass $\mathcal C$ of flag manifolds. Indeed a flag manifold $F$ belongs to $\mathcal C$ when the isotropy representation of the isotropy subgroup $H$ has at most two, necessarily non-equivalent, submodules; the class $\mathcal C$ clearly contains all the Hermitian symmetric spaces as a proper subset as well as infinite series of non-symmetric flag manifolds, as it is shown in Table \eqref{table}, which reproduces the full classification obtained in \cite{AC}. Our result is the following
\begin{theorem}\label{flag2summ} Let $F=K/H$ be a flag manifold with $K$ a compact simple Lie group and let $J$ be an invariant complex structure on $F$. If the isotropy representation of $H$ has at most two summands, then every invariant Hermitian {\rm BTP} metric is either K\"ahler or a multiple of the metric induced by $-B$, where $B$ denotes the Cartan-Killing form of the Lie algebra $\gk$ of $K$.  \end{theorem}
An analogous characterization can also be proved for the class of full flag manifolds for the simple group $\SU(n)$, namely we establish the following
\begin{theorem}\label{SU} An invariant {\rm BTP} Hermitian metric on the flag manifold $F=\SU(n\!+\!1)/\T^n$ endowed with an invariant complex structure is either K\"ahler or a multiple of the invariant metric induced by $-B$, where $B$ denotes the Cartan-Killing form of the Lie algebra $\mathfrak{su}(n\!+\!1)$.
\end{theorem}
We note that a general statement for full flag manifolds of a general compact simple Lie group seems to be out of reach at this very moment. Turning then to the case of a $C$-space, namely the case of a compact simply-connected complex manifold which is acted on transitively a compact group of biholomorphisms (see e.g.\,\cite{A}), we restrict ourselves to the fundamental subclass given by a compact even-dimensional semisimple Lie group endowed with a left-invariant complex structure, a so-called Samelson complex structure. It is well-known that such a compact Lie group fibers holomorphically onto a full flag manifold endowed with an invariant complex structure via the Tits fibration (see \cite{A},\cite{Samelson}) with a maximal torus as typical fibre. These compact homogeneous complex spaces are  particular cases of Wang spaces, namely compact homogeneous spaces $G/K$, where $G$ is a compact semisimple Lie group and $K$ is a compact subgroup whose semisimple part coincides with the semisimple part of the centralizer of some torus in $G$. Given a compact even-dimensional Lie group $G$ endowed with a bi-invariant metric $g_o$ one can consider any Samelson complex structure $J$ w.r.t. which $g_o$ is Hermitian and the corresponding Bismut connection of $(G,g_o,J)$ turns out to have closed and parallel torsion as well as vanishing curvature. Viceversa it was proved in \cite{WYZ} that a compact Hermitian manifold whose Bismut connection is flat is finitely covered by a compact Lie group equipped with a bi-invariant metric and a compatible Samelson structure. It is worth mentioning that compact Lie groups with a Samelson structure provide a distinguished ambient where to find special metrics as SKT or CYT (see e.g.\,\cite{FG}). We provide a characterization of Hermitian BTP metrics on such spaces under the additional hypothesis that the metric is invariant under the right translations by the maximal torus which determines the Samelson structure. More precisely we prove the following
\begin{theorem}\label{Sam} Let $G$ be a compact even-dimensional simple Lie group endowed with a Samelson left-invariant complex structure $J$. Let $g$ be a left-invariant $J$-Hermitian metric on $G$ which is {\rm BTP}. If $g$ descends to an invariant metric $\bar g$  on the flag manifold $Q$ which is the base of the Tits fibration $j:G\to Q$, then $\bar g$ is induced by the Cartan-Killing form of $\gg$ up to a multiple. \par
Viceversa, given the metric $\bar g$ induced on $Q$ by the Cartan-Killing form of $\gg$, then there exists a {\rm BTP} left-invariant $J$-Hermitian metric $g$ on $G$ so that $j:(G,g)\to (Q,\bar g)$ is a Riemannian submersion.
\end{theorem}
In the last section we more specifically address the question when a {\rm BTP} metric turns out to be {\rm BAS} and we first show that the natural Hermitian metrics on isosceles Hopf manifolds and on Calabi Eckmann manifolds are indeed {\rm BAS}. We then provide an explicit example of a compact, non-simply connected homogeneous $4$-fold $M$ carrying a BTP metric which is not  {\rm BAS} as it is shown to be not naturally reductive w.r.t. any transitive action of isometric biholomorphisms. This naturally leads us to formulate the following
\begin{conjecture} \label{conj1.6}
Let $(M^n,g)$ be a compact Hermitian manifold whose universal cover is a homogeneous Hermitian manifold without K\"ahler de Rham factor. Assume $g$ is {\rm BTP}, and the manifold is one of the following
\begin{enumerate}
\item $M^n$ has finite fundamental group, or
\item the universal cover of $(M^n,g)$ is a Lie group equipped with a left-invariant metric and a compatible left-invariant complex structure.
\end{enumerate}
Then $g$ must be {\rm BAS}.
\end{conjecture}
Finally we establish two results to support the conjecture, namely we prove it in some special cases using ad hoc methods. More precisely, in case of $3$-folds we prove
\begin{proposition} \label{prop5.3}
Let $(M^3,g)$ be a compact homogeneous Hermitian manifold without K\"ahler de Rham factor and with finite $\pi_1(M)$. If $g$ is {\rm BTP}, then it is {\rm BAS}.
\end{proposition}
We then consider the case of a nilpotent Lie group with a special left-invariant complex structure, which is called {\it nilpotent} as introduced in \cite{CFGU}, and we obtain the following result
\begin{proposition} \label{prop5.5}
Let $G$ be a nilpotent Lie group equipped with a left-invariant nilpotent complex structure $J$ and a compatible left-invariant metric $g$. If  $g$ is {\rm BTP}, then it is {\rm BAS}.
\end{proposition}
The paper is organized as follows. In \S 2 we recall some basic facts about the Bismut connection of a Hermitian manifold together with some results in case of {\rm BTP} metrics. In \S3 we consider the case of compact {\rm BTP} Chern flat manifolds proving the structure result Theorem \ref{ThmChernflat} together with other results on such metrics, describing their uniqueness and proving that the Bismut curvature is parallel. In \S 4 we focus on the case of a compact simply-connected flag manifold endowed with an invariant {\rm BTP} Hermitian metric. After analyzing this condition in terms of the root system w.r.t. a fixed Cartan subalgebra, we give a proof of Theorems \ref{flag2summ} and \ref{SU}. In \S5 we address the problem of characterizing {\rm BTP} invariant metric on a Wang space, namely on a simply-connected compact complex homogeneous space under the action of a compact Lie group. We prove a structure result, namely Theorem \ref{Sam}, in the case of Wang spaces given by compact even-dimensional Lie groups with a Samelson structure, endowed with special {\rm BTP} metrics. In the last section \S6 we elaborate on the Conjecture \ref{conj1.6} providing examples and proving two results, namely Propositions \ref{prop5.3} and \ref{prop5.5}, which prove the conjecture in two remarkable special cases.

\vspace{0.3cm}
\section{Preliminaries}
Given a Hermitian manifold $(M,g,J)$, a linear connection $\nabla$ is called Hermitian if it leaves both $g$ and $J$ parallel. It is well-known that there exists a distinguished one-parameter family of Hermitian connections $\{{}^t\nabla\}_{t\in\mathbb R}$, firstly introduced by Gauduchon \cite{Gaud} and defined as
$$g({}^t\nabla_XY,Z) = g(\nabla^{\mbox{\tiny LC}}_XY,Z) - \frac{t-1}4 d\o(JX,JY,JZ) - \frac{t+1}4 d\o(JX,Y,Z),$$
where $X,Y,Z$ are vector fields, $\nabla^{\mbox{\tiny LC}}$ denotes the Levi-Civita connection and $\o(X,Y)=g(JX,Y)$ is the associated fundamental $2$-form.
Among these one can find the Chern connection $\nabla^c$ for $t=1$ and the Bismut connection $\nabla^b$ for $t=-1$, also known as the Strominger connection (see \cite{Strominger},\cite{Bismut}), given by
\begin{equation}\label{def}
g(\nabla^b_XY,Z) = g(\nabla^{\mbox{\tiny LC}}_XY,Z) +\frac 12 d\o(JX,JY,JZ).
\end{equation}
The Bismut connection can also be characterized by the property of being the unique Hermitian connection whose torsion tensor $T^b\in\Lambda^2TM^*\otimes TM$ is totally skew-symmetric, namely the tensor $(X,Y,Z)\mapsto g(T^b(X,Y),Z)$ is a $3$-form on $M$.

In this paper we are mainly concerned with the {\rm BTP} condition, namely we consider the case where the torsion $T^b$ is $\nabla^b$-parallel. This condition has been thoroughly investigated in \cite{ZhaoZ19Str}, \cite{ZhaoZ22},  where several important features have been highlighted. In order to describe some of the main properties of {\rm BTP} manifolds, motivating at the same time their study, we select a local unitary frame $\{e_1,\ldots,e_n\}$ of $T^{1,0}(M)$ ($n=\dim_{\mathbb C}M$) and we denote by $T$ the torsion of the Chern connection, so that
$$T(e_i,\overline e_j)=0,\qquad T(e_i,e_j) =  \sum_{k=1}^nT^k_{ij}e_k.$$
Note that the symbol $T^k_{ij}$ in \cite{ZhaoZ19Str} or \cite{ZhaoZ22} is half of our $T^k_{ij}$ here. For the Levi-Civita connection of the metric $g$, we may write $D e_j = \theta^1_{jk}e_k + \overline{\theta^2_{jk}}\overline e_k$ for some $1$-forms $\theta^1,\theta^2$. If we denote by $\theta,\theta^b$ the connection forms under the unitary frame $e$ for the Chern and Bismut connection, respectively, then we have $\theta^b = 2\theta^1-\theta$ (see \cite{WYZ}) and therefore a straightforward computation yields
\begin{equation}\label{Bism}
\nabla^b_{e_i}e_j = 2\nabla^{\mbox{\tiny LC}}_{e_i}e_j -\nabla^c_{e_i}e_j, \qquad  \nabla^b_{\overline e_i}e_j =2\nabla^{\mbox{\tiny LC}}_{\overline e_i}e_j -\sum_k T^i_{jk}\overline{e}_k - \nabla^c_{\overline e_i}e_j.
\end{equation}
From this it follows that the torsion $T^b$ of the Bismut connection can be expressed ad follows
$$ T^b(e_i,e_j) = -\sum_{k=1}^n T^k_{ij}e_k,\qquad T^b(e_i,\overline e_j) = \sum_{k=1}^n\left( T^j_{ik}\overline e_k - \overline{T^i_{jk}}e_k\right),$$
whence one can deduce (see \cite{ZhaoZ19Str}) that
\begin{equation}  \nabla^b T^b =0 \qquad {\rm{iff}}\qquad \nabla^b T = 0.\end{equation}
Again using \eqref{Bism} one can write down the {\rm BTP} condition as follows
\begin{equation}\label{BTP}
\left\{ \begin{split}
 \ \sum_{r=1}^n \big( T^{\ell}_{ri}T^r_{jk} + T^{\ell}_{rj}T^r_{ki} + T^{\ell}_{rk}T^r_{ij}  \big) \ = \ 0  \\
\  \sum_{r=1}^n \big( T^{j}_{ir} \overline{ T^k_{\ell r} } - T^{j}_{kr} \overline{ T^i_{\ell r} } + T^{r}_{ik} \overline{ T^r_{j\ell } } \big) \ = \ 0
 \end{split} \right.
\end{equation}
for any $1\leq i,j,k,\ell \leq n$.

As for the curvature of the connection $\nabla^b$, we first observe that on any Hermitian manifold, the difference  between the Bismut and Chern curvature is given by
\begin{equation}\label{R^b} R^b_{i\bar{j}k\bar{\ell}}- R_{i\bar{j}k\bar{\ell}} = T^{\ell}_{ik,\bar{j}} +\overline{T^k_{j\ell , \bar{i}} } + \sum_r\big( T^{\ell}_{ir} \overline{T^k_{jr} } -  T^{r}_{ik} \overline{T^r_{j\ell } } - T^{j}_{ir} \overline{T^k_{\ell r} } - T^{\ell}_{kr} \overline{T^i_{jr} }  \big)  \end{equation}
under any local unitary frame, where the index after comma stands for covariant derivatives with respect to $\nabla^b$. We remark here that  \eqref{R^b} is formula (3.2) in \cite[Lemma 3.1]{ZhaoZ22} and that the symbol $T^j_{ik}$ there denote half of the Chern torsion components, which explains the discrepancy in coefficients $2$ and $4$ on the right hand side. In particular when the torsion $T^b$ is $\nabla^b$-parallel, we immediately obtain that
\begin{equation}\label{curv}  \nabla^b R^b = 0 \qquad {\rm{iff}}\qquad \nabla^b R =0.\end{equation}
Moreover under the {\rm BTP} condition, it has been proved in \cite{ZhaoZ22} that the curvature $R^b$ has the following remarkable properties, namely
$$R^b(X,Y,Z,\overline W) = 0,\qquad R^b(X,\overline Y,Z,\overline W) = R(Z,\overline W,X,\overline Y)$$
for every $X,Y,Z,W$ tangent vectors of type $(1,0)$. We refer to \cite{ZhaoZ22} for a more detailed analysis of other curvature properties, also in relation to the condition Bismut K\"ahler-like condition.

As a final note, we briefly discuss the case when the curvature $R^b$ vanishes identically. Indeed, it has been proved in \cite{WYZ} that a compact Hermitian manifold whose Bismut connection is flat is locally biholomorphically isometric to an even dimensional Lie group $G$ equipped with a bi-invariant metric $g_o$ and a left-invariant complex structure compatible with $g_o$ (see \cite{WYZ} for a more detailed exposition). As the Bismut connection $\nabla^b$ of $g_o$ satisfies $\nabla^b_XY=0$ for every left-invariant vector field on $G$, we see that Bismut flat Hermitian manifold are actually {\rm BTP}.

\section{Compact Chern flat {\rm BTP} manifolds}

In this section, we  focus on the class of compact Chern flat manifolds and prove a characterization of such manifolds in the case where they can carry a metric whose Bismut connect has parallel torsion.

\begin{proof}[{\bf Proof of Theorem \ref{ThmChernflat}}] Since $(M^n,g)$ is compact Chern flat, by Boothby's theorem \cite{Boothby}, the universal cover  of $M$ is a unimodular complex Lie group $G$ equipped with a compatible left-invariant metric. Denote by $\mathfrak g$ and $\langle , \rangle $ the corresponding Lie algebra and metric. We want to show that $\mathfrak g$ must be the orthogonal direct sum
\begin{equation}
 {\mathfrak g} = {\mathfrak a} \oplus {\mathfrak g}_1 \oplus \cdots \oplus {\mathfrak g}_k \label{eq:decom}
 \end{equation}
where $\mathfrak a$ is the center and each ${\mathfrak g}_i$ is a simple complex Lie algebra.  Note that in particular this says that  $\mathfrak g$ is reductive.

For this purpose, let us denote by $\nabla$, $\nabla^b$  the Chern and Bismut connection and $T$, $T^b$  their torsion tensors. The {\rm BTP}
 condition means that $\nabla^bT^b=0$, which is equivalent to $\nabla^bT=0$ (see \cite{ZhaoZ19Str}, \cite{ZhaoZ22}). Let $e$ be a local unitary $\nabla$-parallel frame of type $(1,0)$ tangent vectors on $M^n$. Lifting to the universal cover, we may extend $e$ into a global frame.
Boothby's theorem says that when $(M^n,g)$  is compact Chern flat one will have $\nabla T=0$, so the torsion components $T^j_{ik}$ under $e$ defined by $T(e_i,e_k)=\sum_j T^j_{ik}e_j$ are constants, and using them as (minus of the) structure constants one can make the universal cover $G$ of $M$ into a complex Lie group, where each $e_i$ is left-invariant. Hence  $e$ becomes a unitary basis of the complex Lie algebra ${\mathfrak g}$, and we have $C^j_{ik}=-T^j_{ik}$ where $C^j_{ik}$ are the structure constants defined by  $[e_i,e_k]=\sum_j C^j_{ik} e_j$.
 If we now use formulas \eqref{BTP}, we note that the first one, in terms of $C$, is just the Jacobi identity, so we will focus on the second one. Following a trick in  \cite{ZhaoZ22}, in the second equation of \eqref{BTP}, we interchange $i$ and $j$, and also interchange $k$ and $\ell$, and then take the complex conjugation, it gives us
$$ \sum_{r=1}^n \big( T^{\ell}_{kr} \overline{ T^i_{jr} } - T^{j}_{kr} \overline{ T^i_{\ell r} } + T^{r}_{ik} \overline{ T^r_{j\ell } } \big) \ = \ 0 .$$
Subtracting that from the second equation of \eqref{BTP}, we obtain
$$ \sum_{r=1}^n \big(  T^{j}_{ir} \overline{ T^k_{\ell r} } -  T^{\ell}_{kr} \overline{ T^i_{jr} }  \big) = 0, \ \ \ \ \forall \ 1\leq i,j,k,\ell \leq n. $$
Choosing $k=j$ and $\ell =i$, we get
\begin{equation}
\sum_r |T^j_{ir}|^2 = \sum_r |T^i_{jr}|^2 , \ \ \ \ \ \forall \ 1\leq i,j\leq n. \label{eq:BTP2}
\end{equation}

Of course the same equality holds when we replace $T$ by $C$, as $C=-T$. Also, the equality holds for any choice of unitary basis of ${\mathfrak g}$.

Now suppose ${\mathfrak h}$ is an ideal of the complex Lie algebra ${\mathfrak g}$. Denote by ${\mathfrak h}^{\perp}$ its orthogonal complement in ${\mathfrak g}$. Take a unitary basis $e=\{ e_1, \ldots , e_n\}  $ of ${\mathfrak g}$ so that $\{ e_1, \ldots , e_p\} $ is a basis of ${\mathfrak h}$. Since ${\mathfrak h}$ is an ideal, we have $[{\mathfrak h},{\mathfrak g}]\subseteq {\mathfrak h}$, so
$$ C^{\alpha}_{i\ast} =0 , \ \ \ \ \ \forall \ 1\leq i\leq p, \ \ \forall \ p+1\leq \alpha \leq n.$$
By (\ref{eq:BTP2}), we have $C^i_{\alpha \ast}=0$ as well, so  $[{\mathfrak h}^{\perp},{\mathfrak g}]\subseteq {\mathfrak h}^{\perp}$, thus we have proved that:

{\em For any ideal ${\mathfrak h}$ of ${\mathfrak g}$, its orthogonal complement ${\mathfrak h}^{\perp}$ is also an ideal in ${\mathfrak g}$. }

This establishes the decomposition (\ref{eq:decom}) and completes the proof of the first part of Theorem \ref{ThmChernflat}.

Conversely, let $G$ be a direct product $G={\mathbb C}^k \times G_1\times \cdots \times G_r$ where each $G_i$ is a connected, simply-connected simple complex Lie group. It suffices to deal with the semisimple case, so in the following we will assume that $G$ is a semisimple complex Lie group, with Lie algebra ${\mathfrak g}$. If we denote by $\mathfrak g_{\mathbb R}$ the Lie algebra $\mathfrak g$ considered as a real algebra, then a Cartan decomposition of  $\mathfrak g_{\mathbb R}$ is given by ${\mathfrak g_{\mathbb R}} = {\mathfrak u} + J{\mathfrak u}$, where ${\mathfrak u}$ is the Lie algebra of a compact semisimple Lie group $U$. If $B$ denotes the Cartan-Killing form of $\gg_{\mathbb R}$ we define the Hermitian metric $g$  by putting
$$g|_{\gu\times\gu} = -B,\ g(\gu,J\gu)=0,\ g|_{J\gu\times J\gu} = B.$$
This metric depends on the choice of a compact real form $\gu$, but as two such real forms are conjugate by an inner automorphism of $\gg$ (see e.g. \cite{He}, p.184), the corresponding metrics are biholomorphically isometric. This construction is well-known (see e.g. \cite{He}, p.182) and these metrics were used in \cite{LPV}, where they are called {\it canonical} metrics, and in \cite{G} among others. \par
In order to compute the Bismut connection using formula \eqref{def}, we identify elements of $\gg$ with left-invariant vector fields on $G$ and we have the following well-known expressions
$$2 \langle \nabla^{\mbox{\tiny LC}}_xy,z\rangle = \langle[x,y],z\rangle + \langle [z,x],y\rangle + \langle [z,y],x\rangle,$$
while
$$d\omega(x,y,z) = \langle [x,y],Jz\rangle +\langle[z,x],Jy\rangle +\langle [y,z],Jx\rangle, $$
where $J\in {\rm{End}}(\gg)$ denotes the complex structure. Using the fact that ${\rm ad}(x)$ commutes with $J$ for every $x\in \gg$, we have
$$d\omega(Jx,Jy,Jz) = \langle [x,y],z\rangle +\langle[z,x],y\rangle +\langle [y,z],x\rangle, $$
so that
$$\langle \nabla^b_xy,z\rangle = \langle [x,y],z\rangle + \langle [z,x],y\rangle,\qquad x,y,z\in\gg.$$
This can also be written as
$$ \nabla^b_x = \ad(x) - \ad(x)^*,\quad x\in \gg$$
where ${}^*$ denotes the adjoint with respect to $g$. \par

Using the fact that ${\rm ad}(x)$ is skew-symmetric if $x\in\gu$ and symmetric if $x\in J\gu$, we obtain that
\begin{eqnarray}
&& \nabla^b_x = 2\ \ad(x),\ x\in\gu,\\
&&\label{connection}\nabla^b_x = 0,\ x\in J\gu.
\end{eqnarray}
Now we consider the torsion $T^b$ which is given by
\begin{eqnarray}
&& T^b(y,z) = 3\ [y,z],\quad y,z\in\gu,\\
&& T^b(y,z)= [y,z],\quad y\in \gu,\ z\in J\gu,\\
&&T^b(y,z) = -[y,z],\quad y,z\in J\gu.
\end{eqnarray}
that $\nabla^bT^b=0$, we only need to verify $(\nabla^b_xT^b)(y,z) = 0$ for $x\in \gu,\ y,z\in \gg$ by formula \eqref{connection}.
When $y,z\in \gu$ we have
$$(\nabla^b_xT^b)(y,z)= 3\nabla^b_x[y,z] - T^b(2[x,y],z )- T^b(y, 2[x,z])$$
$$ = 6 \left([x,[y,z]] - [[x,y],z] - [y,[x,z]]\right) = 0$$
by Jacobi identity. When $z\in J\gu$ and $y\in \gu$ ($y\in J\gu$ resp.), again by Jacoby identity
$$(\nabla^b_xT^b)(y,z) = 2\epsilon\, ([x,[y,z]] - [[x,y],z] - [y,[x,z]]) = 0$$
with $\epsilon =1$ ($\epsilon=-1$ resp.).

\end{proof}

\begin{remark}
For $(M^n,g)$ in Theorem \ref{ThmChernflat}, we have $M=G/\Gamma$ where $\Gamma$ is a discrete subgroup of $\,\mbox{Iso}_{h}(G,g)$, the group of holomorphic isometries of $(G,g)$. The latter contains $G$ (as left translations) since $g$ is left-invariant. Note that in general there might not be a finite index subgroup $\Gamma'\subseteq \Gamma$ such that $\Gamma'\subseteq G$, and there might not be a finite index subgroup $\Gamma'\subseteq \Gamma$ such that $\Gamma'$ is contained in the product of holomorphic isometry groups of the factors of $G$.
\end{remark}

We recall that a metric connection with parallel torsion and curvature on a Riemannian manifold $(M,g)$ is called an {\em Ambrose-Singer} connection, and its existence is a
necessary and sufficient condition for $M$ to be locally homogeneous (see \cite{AS},\cite{TV}). Moreover when $(M,g)$ is simply-connected and complete, it admits an Ambrose-Singer connection with totally skew-symmetric torsion if and only if it is a {\em naturally reductive} homogeneous space with respect to some transitive Lie group $G$ of isometries (see e.g.\,\cite{TV}). By the work of Sekigawa \cite{Sekigawa} in the Hermitian case, a complete, simply-connected Hermitian manifold $(M,g,J)$ is homogeneous if and only if there exists an Ambrose-Singer connection which is Hermitian. Therefore $(M,g,J)$ is a  naturally reductive homogeneous space if and only if its Bismut connection is an Ambrose-Singer connection. As introduced in Definition \ref{defAS}, a {\rm BTP} metric $g$ on an Hermitian manifold $(M,g,J)$ is called {\em Bismut Ambrose-Singer} ({\rm BAS} in short) if its Bismut connection has parallel curvature or equivalently, if its Bismut connection is an Ambrose-Singer connection.

It can be verified that for the metric $g$ constructed in Theorem \ref{ThmChernflat}, its Bismut connection also has parallel curvature thus is Ambrose-Singer. We have the following observation:

\begin{proposition}\label{ChernflatBAS}
Let $(M^n,g)$ be a compact Hermitian manifold that is Chern flat and {\rm BTP}. Then its Bismut connection is Ambrose-Singer, hence its universal covering manifold is naturally reductive.
\end{proposition}

\begin{proof} Using formula \eqref{R^b}  we see that when the metric $g$ is Chern flat and {\rm BTP}, $R^b$ is given by a quadratic expression of $T$, hence $\nabla^bR^b=0$. (Note that the components $R^b_{ijk\bar{\ell}}$ always vanish for {\rm  BTP} metrics).
\end{proof}

Next let us examine the uniqueness issue, namely, how many Chern flat {\rm BTP} metrics are there on the same manifold?
Let $M^n$ be a compact complex manifold. Suppose $g$, $g'$ are Chern flat and {\rm BTP} metrics on $M^n$. Then the torsion components of $g$ and $g'$ turn the universal cover $\widetilde{M}$ into (reductive) complex Lie groups $G$ and $G'$. First we claim the following

\begin{proposition}
Let $M^n$ be a compact complex manifold. Suppose $g$, $g'$ are both Chern flat {\rm BTP} metrics on $M^n$. The universal cover becomes reductive complex Lie groups $G$ and $G'$, respectively. Then $G$ is isomorphic to $G'$.
\end{proposition}

\begin{proof}
Note that $G$ and $G'$  are reductive simply-connected homeomorphic complex Lie groups. Consider their maximal compact subgroups $K$, $K'$, which are semisimple and simply-connected. Clearly $K$ and $K'$ share the same topological invariants and therefore are isomorphic as Lie groups. It then follows that the semisimple parts of $G$  and $G'$ (which are the complexifications of $K$, $K'$) are isomorphic. The abelian part then must also be isomorphic. Thus we have proved that $G$ is isomorphic to $G'$.
\end{proof}

The remaining question is to know how many left-invariant (thus Chern flat) {\rm BTP} metrics are there on the same reductive complex Lie group $G$. By (\ref{eq:decom}), we know that any left-invariant {\rm BTP} metric $g$ on $G$ must preserve the product structure, and it is easy to see that the restriction of $g$ on the vector group factor is K\"ahler and flat. Thus the main issue here rests upon those simple factors.

Suppose $G$ is a simple complex Lie group with Lie algebra ${\mathfrak g}$. Denote by $B$ the Cartan-Killing form of ${\mathfrak g}$. Then a {\em $B$-isometry} is a linear isomorphism $f\in GL({\mathfrak g})$ such that $B(f(X),f(Y))=B(X,Y)$ for any $X$, $Y\in {\mathfrak g}$. Denote by $\mbox{Iso}(B)$ the group all $B$-isometries of ${\mathfrak g}$. In general it is bigger than $\mbox{Aut}({\mathfrak g})$.

If $g$ and $h$ are both left-invariant Hermitian metrics on $G$, then they determine a unique linear isomorphism $f\in GL({\mathfrak g})$ such that $h(X, Y) = g(f(X), Y)$ for any $X$, $Y\in {\mathfrak g}$. We have the following

\begin{theorem} \label{thm2.6}
Let $G={\mathbb C}^k \times G_1\times \cdots \times G_r$ where each $G_i$ is a connected, simply-connected simple complex Lie group. Then any left-invariant {\rm BTP} Hermitian metric on $G$ preserves the product structure, is K\"ahler flat on the Euclidean factor, and on each simple factor, two such metrics differ by a constant multiple of a $B$-isometry.
\end{theorem}

\begin{proof} Let $G$ be as above and $g$ a left-invariant Hermitian metric on $G$ that is {\rm BTP}. By (\ref{eq:decom}), we know that $g$ must be a product metric. Clearly the vector group factor corresponds to the kernel of the torsion $T$ which coincides with the kernel of the Bismut curvature $R^b$, hence is K\"ahler and flat. The main issue here is to determine the metric up to a constant multiple on each simple factor. So in the following let us assume that $G$ is a (connected and simply-connected) simple complex Lie group and $g$ a left-invariant {\rm BTP} Hermitian metric on $G$.

We want to understand the relationship between $g$ and the group structure of $G$.  Consider the Cartan-Killing form $B$ of $\gg$, which can be written as
\begin{equation} \label{eq:Killing}
B(X,Y) = \mbox{Tr}(\ad(X)\circ\ad(Y)) =  \sum_{r,s=1}^n T^r_{sX} T^s_{rY}.
\end{equation}
As $T$ is $\nabla^b$-parallel, we have $\nabla^bB =0$, so by using formulas \eqref{Bism} we obtain for any $1\leq i,j\leq n$ that
\begin{equation}\label{parB}
\left\{ \begin{split} \  \sum_{l=1}^n \big(  T^l_{ki}B(e_l,e_j) + T^l_{kj}B(e_i,e_l) \big) = 0,\\
\ \sum_{i=1}^l\big( \overline{T^i_{kl}}B(e_l,e_j) + \overline{T^j_{kl}}B(e_i,e_l)\big)  =0.
\end{split} \right.
\end{equation}
It is well known that for any given symmetric complex matrix $P$, there always exists a unitary matrix $Q$ such that $^t\!Q P Q $ is diagonal with non-negative real elements. So by a constant unitary change of $e$ if necessary, we may assume that $B= (B_{ij}) =\mbox{diag} \{ a_1, \ldots , a_n\}$ with $a_1\geq \cdots \geq a_n\geq 0$. Therefore \eqref{parB} are written as
\begin{equation*}
a_jT^j_{ik} + a_iT^i_{jk} =0 , \ \ \ \ a_jT^i_{jk} + a_iT^j_{ik} =0 , \ \ \ \ \forall \ 1\leq i,j,k\leq n.
\end{equation*}
From this we obtain for $1\leq i,j\leq n$
$$a_j^2\sum_r |T^j_{ir}|^2 = a_i^2 \sum_r |T^i_{jr}|^2 = a_i^2 \sum_r |T^j_{ir}|^2.$$
We then conclude that
\begin{equation} \label{eq:SPT3}
T^j_{i\ast} = T^i_{j\ast} =0  \ \ \ \mbox{whenever} \ a_i\neq a_j.
\end{equation}
Now suppose that the diagonal elements of $B$ are not all equal, say $a_1=\cdots =a_r > a_{r+1} \geq \cdots \geq a_n\geq 0$. Then for any $1\leq i\leq r$ and any $r+1\leq \alpha \leq n$, by (\ref{eq:SPT3}), we know that $T^i_{\alpha \ast } = T^{\alpha}_{i\ast}=0$. So if we let
$${\mathfrak g}_1={\mathbb C}\{ e_1, \ldots , e_r\}, \ \ \ \  {\mathfrak g}_2={\mathbb C}\{ e_{r+1}, \ldots , e_n\}, $$
then both ${\mathfrak g}_1$ and ${\mathfrak g}_2$ are ideals in the complex Lie algebra ${\mathfrak g}$, contradicting with the assumption that ${\mathfrak g}$ is simple. Therefore the diagonal matrix $B$ must be a constant multiple of the identity: $B=a_1I$.  As ${\mathfrak g}$ is a semisimple complex Lie algebra, we must have $a_1>0$, therefore
$$ B(e_i,e_j) =  a_1 \delta_{ij}. $$
Note that if  $\tilde{e}$ is another basis of ${\mathfrak g}$, with $\tilde{e}_i=\sum_{j=1}^nP_{ij}e_j$, then
$$ B(\tilde{e}_i,\tilde{e}_j) =  a_1 \sum_k P_{ik}P_{jk} , \ \ \ \mbox{while} \ \ g(\tilde{e}_i,\tilde{e}_j) = \sum_k P_{ik}\overline{P_{jk}} , $$
so we get the following relation between $g$ and $B$:
$$
B \,^t\!g^{-1} \overline{B} = a_1^2 g,
$$
which holds under any basis of ${\mathfrak g}$. Now suppose $h$ is another left-invariant Hermitian metric on $G$ that is {\em {\rm BTP}}. Then the above equation would hold for $h$ as well:
$$ B \,^t\!h^{-1} \overline{B} = a_1'^2h,$$
where $a_1'>0$ is another constant. Let us denote by $f: {\mathfrak g} \rightarrow {\mathfrak g}$ the isomorphism (as complex linear map) so that $h(X,Y)=\frac{a_1}{a_1'} g(f(X), Y)$ for any $X,Y \in {\mathfrak g}$. Then the above relations imply that
\begin{equation} \label{Bf}
 B(f(X), f(Y)) = B(X,Y) , \ \ \ \ \forall \ X, Y\in {\mathfrak g}.
\end{equation}
So a constant multiple of $h$ is related to $g$ by a $B$-isometry. This completes the proof of Theorem \ref{thm2.6}.
\end{proof}

\begin{remark}
Given a simple complex Lie group $G$, Theorem \ref{ThmChernflat} says that there exist left-invariant {\rm BTP} metrics on $G$, and Theorem \ref{thm2.6} says that such metrics are related by $B$-isometries up to constant multiples. We do not know if such metrics are always related by inner automorphisms (up to constant multiples) or not, although this is indeed the case sometimes, for instance, when $n=3$.
\end{remark}

\vspace{0.3cm}

\section{Flag manifolds and the {\rm BTP} condition}

In this section we consider the case of a compact simply-connected (generalized) flag manifold, namely a compact simply-connected K\"ahler homogeneous manifold $F$. We first recall the construction of invariant complex structures on $F$ and then we will investigate the   properties of {\rm BTP} metrics.  \par
Given a compact connected semisimple Lie group $K$ and a closed subgroup $H\subset K$, a homogeneous space $F:= K/H$ is called a {\it flag manifold} if $H$ is the centralizer in $K$ of a torus $Z\subset K$. In this case $H$ turns out to be connected and it contains a maximal torus $T$ of $K$. At the level of Lie algebras, we may fix a reductive decomposition $\gk = \gh + \gm$, with $[\gh,\gm]\subseteq \gm$ so that $\gm$ identifies with the tangent space $T_{[H]}F$. This identification is obtained using the map $\imath:X\ni\gk\to X^*\in\Gamma(TF)$ which associates to each element $X$ of $\gk$ the vector field $X^*|_{gH} = \frac{d}{dt}|_{t=0}\exp(tX)gH$. Note that $[X^*,Y^*]= - [X,Y]^*$ for every $X,Y\in \gk$.

It is well-known (see e.g.\,\cite{BFR}) that an invariant almost complex structure $J$ on $F$ corresponds to an endomorphism $J\in \mbox{End}(\gm)$ with $J^2=-Id$ and $[\mbox{Ad}(h),J]=0$ for every $h\in H$. If we decompose $\gm^{\mathbb C} = \gm^{10}+\gm^{01}$ into $J$-eigenspaces, then the integrability of $J$ is equivalent to the condition
$$[\gm^{10},\gm^{10}]_{\gm^{\mathbb C}}\subseteq \gm^{10},$$
where the subscript ${-}_{\gm^{\mathbb C}}$ denotes the projection onto $\gm^{\mathbb C}$ with respect to the decomposition $\gk^{\mathbb C} = \gh^{\mathbb C} + \gm^{\mathbb C}$. Using the $\mbox{Ad}(H)$-invariance of $J$, the integrability condition can also be expressed by saying that $\gh^{\mathbb C}+ \gm^{10}$ is a subalgebra.

If we fix a maximal torus $T\subseteq H$ and therefore a maximal abelian subalgebra $\gt$ of $\gk$, then we can consider the root space decomposition
$$\gk^{\mathbb C} = \gt^{\mathbb C} + \bigoplus_{\alpha\in R}\gk_\alpha,$$
where $R\subset(\gt^{\mathbb C})^*$ is the set of all roots and $\gk_\alpha = \{v\in \gk^{\mathbb C}|\ [h,v]=\alpha(h)v,\ h\in \gt^{\mathbb C}\}$ is the root space corresponding to the root $\alpha\in R$. As $\gh$ contains $\gt$, we have that $\gh^{\mathbb C} = \gt^{\mathbb C} + \bigoplus_{\alpha\in R_{\gh}}\gk_\alpha$ for some subset $R_\gh$ of $R$ and $\gm^{\mathbb C} = \bigoplus_{\alpha\in R_\gm}\gk_\alpha$, where $R_\gm = R\setminus R_{\gh}$. Moreover as $J$ commutes with $\ad(h)$ for every $h\in \gt^{\mathbb C}$, we have that $\gm^{10}$ is made of root spaces. Then we can select an ordering of the roots, namely a splitting $R = R^+ \cup R^-$ with $R^-= - R^+$ so that for $R_\gm^{\pm} = R_\gm \cap R^{\pm}$ we have
$$(R_\gh + R_\gm^+) \cap R \subseteq R_\gm^+,\quad (R_\gm^+ + R_\gm^+) \cap R \subseteq R_\gm^+$$
and
$$J|_{\gk_\alpha} =  i\epsilon_\a\ \mbox{Id}\ \quad \mbox{where}\qquad \epsilon_\a=\pm 1\ \mbox{for}\ \alpha\in R_\gm^{\pm}.$$
This means that $\gm^{10} = \bigoplus_{\alpha\in R_\gm^+} \gk_\alpha$.

We now consider an invariant $J$-Hermitian metric $g$. The $\ad(\gt)$-invariance of $g$ implies that if $v\in \gk_\alpha$ and $w\in \gk_\beta$ for $\alpha,\beta\in R_\gm$ then
$$0 = (\alpha(h)+ \beta(h)) \cdot g(v,w),$$
so that $g(v,w)\neq 0 $ implies $\alpha = -\beta$. This means that $g$ is completely determined by its restriction to the one-dimensional root spaces $\gk_\alpha$, $\alpha\in R_\gm^+$. In order to simplify computations, we fix the Weyl basis $\{E_\alpha\}$ of $\gk_\alpha$ for $\alpha\in R$ with the property that (see \cite{He}, p.176)
$$B(E_\alpha,E_{-\alpha}) = 1, \quad [E_\alpha,E_{-\alpha}] = H_\alpha,$$
where $\alpha(h) = B(h,H_\alpha)$ for every $h\in \gt^{\mathbb C}$. The real algebra $\gk$ is spanned by $\gt$ together with the vectors
$$v_\alpha := E_\alpha-E_{-\alpha},\quad w_\alpha:= i\ (E_\alpha + E_{-\alpha}),\quad \alpha\in R_\gm^+$$
and $\overline{E_\alpha}= -E_{-\alpha}$ for every $\alpha \in R_\gm$,where the conjugation is with respect to the compact real form $\gk$ of $\gk^\mathbb C$.
It then follows that the metric $g$ satisfies $g(E_\alpha,E_\beta) \neq 0$ if and only if $\beta=-\alpha$ and therefore it is completely determined by the real positive numbers
$$g_\alpha := g(E_\alpha,\overline{E_{\alpha}}) = - g(E_\a,E_{-\a}) = \frac 12 g(v_\alpha,v_\alpha) > 0$$
which are subject to the condition of $\mbox{ad}(\gh)$-invariance, namely
$$g_{\alpha+\gamma} = g_\alpha,\qquad \alpha\in R_\gm,\ \gamma\in R_\gh,\ \alpha+\gamma\in R.$$
It is not difficult to see that the metric $g$ is K\"ahler if and only if the coefficients $\{g_\alpha\}_{\alpha\in R_\gm^+}$ satisfy the relation
$$g_{\alpha+\beta} = g_\alpha + g_\beta,\quad \alpha,\beta,\alpha+\beta\in R_\gm^+.$$

A special Hermitian metric is given by $g_o:= - B|_{\gm \times \gm}$, where $B$ denotes the Cartan-Killing form of $\gk$, and this corresponds to the condition $g_\alpha = 1$ for every $\alpha\in R_\gm$, so that $g_o$ is K\"ahler if and only if the sum of two positive roots in $R_\gm^+$ is never a root , hence $[\gm,\gm]_\gm=\{0\}$, i.e. $K/H$ is a Hermitian symmetric space. We now show that the associated Bismut connection coincides with the canonical connection on the reductive homogeneous space $K/H$. Indeed we can consider the invariant canonical linear connection $\nabla$ on $K/H$ whose torsion tensor $T^\nabla$ at the point $[eH]\in K/H$ is given by (see e.g.\,\cite{KN}, p.193)
$$T^\nabla(X,Y)|_{[eH]} = - [X,Y]_\gm\qquad X,Y\in \gm.$$
As any $K$-invariant tensor on $K/H$ is $\nabla$-parallel, then $\nabla$ leaves the metric $g_o$ and $J$ parallel. Moreover, by the standard property of $B$, we have at $[eH]$
$$g_o(T^\nabla(X,Y),Z) = B([X,Y]_\gm,Z) = B([X,Y],Z) = - B([X,Z],Y) = -g_o(T^\nabla(X,Z),Y)$$
for any $X,Y,Z\in\gm$, so that the torsion is skewsymmetric and $\nabla$ coincides with the Bismut connection $\nabla^b$ of $g_o$. Note the curvature $R^\nabla$ is also $\nabla$-parallel, so that $\nabla$ is an Ambrose-Singer connection of $(K/H,g_o)$. In other words, $(K/H,g_o)$ is a {\em BAS} manifold in the terminology of \cite{NiZheng}, where the word is the abbreviation for {\em Bismut Ambrose-Singer}, meaning a compact Hermitian manifold whose Bismut connection is Ambrose-Singer. To summarize, we have

\begin{proposition}
The Cartan-Killing metric $g_o$ on the flag manifold $K/H$ is {\rm BTP}. In fact it is {\rm BAS}.
\end{proposition}

 We remark that in \cite{ZhaoZ22}, where the classification of three-dimensional compact balanced {\em {\rm BTP}} manifolds were given, the Fano case leads to the flag threefold, and it took a somewhat length computation to verify that the particular metric is balanced and {\em {\rm BTP}}. By the above proposition it suffices to check that the metric is associated with the Cartan-Killing form.

The following proposition shows that all the invariant Hermitian metrics on a flag manifold are balanced. This result is probably already known, but we prefer to include an elementary proof.
\begin{proposition} Every $K$-invariant Hermitian metric on $K/H$ is balanced.
\end{proposition}
\begin{proof} Let $g$ be any invariant Hermitian metric. The balanced condition is given by $d\omega^{n-1}=0$, where $\omega$ is the associated fundamental form of $g$ and $n=\dim_{\mathbb C}K/H$. This is equivalent to saying that $\delta\omega=0$, where $\delta$ denotes the codifferential w.r.t. $g$. Now, $\delta\omega$ is an invariant $1$-form, hence it corresponds to an invariant vector field $X$ on $K/H$. This implies that the vector $X_{[eH]}$ corresponds to a vector $v\in\gm$ which is invariant by the isotropy representation $\lambda:H\to \mbox{SO}(\gm,g)$, hence $[h,v]=0$ for every $h\in \gh$. In particular $[v,\gt]=0$, contradicting the maximality of the abelian subalgebra $\gt$, unless $v=0$, hence $X=0$ by invariance. \end{proof}
We now consider an invariant Hermitian metric $g$ on the flag manifold $F= K/H$ and compute its Bismut connection. We recall that the Nomizu operator relative to the Levi-Civita connection $\nabla^{\mbox{\tiny LC}}$ of $g$, namely $\Lambda^{\mbox{\tiny LC}}:\gm\to\so(\gm)$ defined as $\Lambda^{\mbox{\tiny LC}}(X)(Y) = \imath|_{\gm}^{-1}(\nabla^{\mbox{\tiny LC}}_{X^*}Y^*)|_{[H]}+[X,Y]_\gm$, is given by (see \cite{KN})
$$\Lambda^{\mbox{\tiny LC}}(X)(Y) = \frac 12 ([X,Y]_\gm + U(X,Y)),$$
where the bilinear map $U:\gm\times\gm\to\gm$ is defined as
$$g(U(X,Y),Z) = g([Z,X]_\gm,Y) + g([Z,Y]_\gm,X).$$
From this definition is follows that for roots $\a,\b\in R_\gm$, we have that $g(U(E_\a,E_\b),E_\g)\neq 0$ only when $\g=-\a-\b$, provided this is a root. If we use the standard notation $[E_\a,E_\b]=N_{\a,\b}E_{\a+\b}$, we have
$$g(U(E_\a,E_\b),E_{-\a-\b}) = -N_{-\a-\b,\a}\cdot g_\b - N_{-\a-\b,\b}\cdot g_\a.$$
Using the property that $N_{-\a-\b,\a}=N_{\a,\b} = N_{\b,-\a-\b}$ (see e.g.\,\cite{He}, p.172), we obtain
$$U(E_\a,E_\b) = N_{\a,\b} \frac{g_\b - g_\a}{g_{\a+\b}}(E_{\a+\b})_\gm,\quad \a,\b \in R_\gm$$
and therefore
\begin{equation}\label{LambdaLC}\Lambda^{\mbox{\tiny LC}}(E_\a)(E_\b) = \frac 12 N_{\a,\b}\left( 1 + \frac{g_\b-g_\a}{g_{\a+\b}}\right)\cdot (E_{\a+\b})_\gm.\end{equation}
In order to compute the Bismut connection, we need the expression of $d\omega(E_\a,E_\b,E_\g)$ for $\a,\b,\g\in R_\gm$. Using the $\ad(\gt)$-invariance, we see that $d\omega(E_\a,E_\b,E_\g)=0$ if $\g\neq -\a-\b$ , while we have (see also \cite{AD})
$$d\omega(E_\a,E_\b,E_{-\a-\b}) = -i N_{\a,\b}\left(\epsilon_\a g_\a + \epsilon_\b g_\b - \epsilon_{\a+\b}g_{\a+\b}\right)  $$
and
$$d\omega(JE_\a,JE_\b,JE_{-\a-\b}) = - N_{\a,\b}\left(\epsilon_\a \epsilon_\b g_{\a+\b} - \epsilon_\a \epsilon_{\a+\b}g_\b - \epsilon_\b\epsilon_{\a+\b}g_{\a}\right).  $$
Using formula (\ref{def}) for the Bismut connection, we see that the corresponding Nomizu operator $\Lambda^b$ is given by
$$\Lambda^b(E_\a)E_\b = \Lambda^{\mbox{\tiny LC}}(E_\a)(E_\b) -\frac 1{2g_{\a+\b}}d\omega(JE_\a,JE_\b,JE_{-\a-\b})\, (E_{\a+\b})_\gm=$$
$$= \frac 12 N_{\a,\b}\left( (1+\ep_\a\ep_\b) - (1+\ep_\b\ep_{\a+\b})\frac{g_\a}{g_{\a+\b}} + (1-\ep_\a\ep_{\a+\b})\frac{g_\b}{g_{\a+\b}}\right) \, (E_{\a+\b})_\gm.$$
In particular we have
\begin{eqnarray}\label{conn}
&& \Lambda^b(E_\a)(E_\b)= N_{\a,\b}\left(1-\frac{g_\a}{g_{\a+\b}}\right) \, E_
{\a+\b},\qquad \mbox{if}\ \a,\b\in R_\gm^+,\\
&& \Lambda^b(E_\a)(E_\b)= 0, \qquad \mbox{if}\ \a\in R_\gm^+, \b\in R_\gm^-, \a+\b\not\in R\ {\rm{or}}\ \a+\b\in R_\gm^+\\
&& \Lambda^b(E_\a)(E_\b)= N_{\a,\b}\left(\frac{g_\b-g_\a}{g_{\a+\b}}\right) \, E_
{\a+\b},\qquad \mbox{if}\ \a\in R_\gm^{\pm},\b\in R_\gm^{\mp}, \a+\b \in R_\gm^{\mp}.
\end{eqnarray}
Therefore the torsion $T^b$, which is given by $T^b(X,Y) = \Lambda^b(X)(Y)-\Lambda^b(Y)(X) - [X,Y]$ for $X,Y\in \gm$, can be written as
\begin{eqnarray}
&& T^b(E_\a,E_\b) = N_{\a,\b}\left( 1 - \frac{g_\a + g_\b}{g_{\a+\b}}\right) E_{\a+\b},\qquad \mbox{if}\ \a,\b\in R_\gm^+,\\
&& T^b(E_\a,E_\b) = 0, \qquad \mbox{if}\ \a\in R_\gm^+,\b\in R_\gm^-, \a+\b\not\in R_\gm,\\
&& T^b(E_\a,E_\b) = N_{\a,\b}\left( \frac{g_\a - g_\b}{g_{\a+\b}}-1\right) E_{\a+\b},\qquad \mbox{if}\ \a\in R_\gm^+,\b\in R_\gm^-, \a+\b\in R_\gm^+,\\
&& T^b(E_\a,E_\b) = N_{\a,\b}\left( \frac{g_\b - g_\a}{g_{\a+\b}}-1\right) E_{\a+\b},\qquad \mbox{if}\ \a\in R_\gm^+,\b\in R_\gm^-, \a+\b\in R_\gm^-.
\end{eqnarray}
Using the skew-symmetry of the torsion $T^b$, we immediately see that the {\em {\rm BTP}} condition $\nabla^bT^b=0$ boils down to verifying that for every positive roots $\a,\b\in R_\gm^+$ we have
$$\nabla^b_{E_\gamma}T^b(E_\a,E_\b) = 0, \qquad \mbox{for\ every}\ \gamma\in R_\gm.$$
\begin{lemma}\label{Lemma1} Let $\a,\b\in R_\gm^+$ with $\a+\b\in R_\gm^+$. Given the two conditions
\begin{itemize}
\item[i)] $\nabla^b_{E_{-\a}}T^b(E_\a,E_\b)=0$
\item[ii)] $\nabla^b_{E_{-\b}}T^b(E_\a,E_\b)=0$
\end{itemize}
we have:
\begin{itemize}
\item[1.] if $\a-\b\not\in R_\gm$ then (i) and (ii) imply that  either $g_{\a+\b} = g_\a + g_\b$ or $g_{\a+\b} = g_\a=g_\b$;
\item[2.] if  $\a-\b\in R_\gm^+$, then (i) implies that $g_{\a+\b} = g_\a + g_\b$ or $g_{\a+\b} = g_\a$;
\item[3.] if $\a-\b\in R_\gm^-$, then (ii) implies that either $g_{\a+b} = g_\a + g_\b$ or $g_{\a+\b} = g_\b$.
\end{itemize}
\end{lemma}
\begin{proof} We have that (i) is equivalent to
$$0 = \Lambda^b(E_{-\a})(T^b(E_\a,E_\b)) - T^b(\Lambda^b(E_{-\a})E_\a,E_b) - T^b(E_\a,\Lambda^b(E_{-\a})E_\b),$$
hence
$$0= N_{\a,\b}\left( 1 - \frac{g_\a + g_\b}{g_{\a+\b}}\right) \Lambda^b_{E_{-\a}}(E_{\a+\b}) - T^b(E_\a,\Lambda^b(E_{-\a})E_\b) =  $$
$$ = N_{\a,\b}\left( 1 - \frac{g_\a + g_\b}{g_{\a+\b}}\right)\cdot N_{-\a,\a+\b}\left(\frac{g_{\a+\b}-g_\a}{g_{\b}}\right)E_\b - T^b(E_\a,\Lambda^b(E_{-\a})E_\b).$$
In case 1 and 2, we have that $\Lambda^b(E_{-\a})E_\b=0$ and therefore either $g_{\a+\b} = g_\a + g_\b$ or $g_{\a+\b} = g_\a$. Repeating the same argument using (ii), we get our claim.\par
In case 3, we use (ii) and the same computation as above with $\a,\b$ interchanged gives our claim. \end{proof}
\par
We are now able to characterize {\em {\rm BTP}} metrics on a certain class $\mathcal C$ of flag manifolds, where the isotropy representation has at most two irreducible summands, as mentioned in the Introduction.
\begin{proof}[\bf{Proof of Theorem \ref{flag2summ}}]  We start noting that in case the flag manifold is isotropy irreducible, then every invariant metric is a multiple of the metric induced by $-B$ and our claim follows. Therefore we focus on the case when the isotropy representation of $H$ has two irreducible summands, say $\gm = \gm_1\oplus \gm_2$. These flag manifolds have been classified in \cite{AC} and we recall here some facts which are relevant to our discussion. In particular we can choose a system of simple roots of $\gk$, say $\Pi=\{\a_1,\ldots,\a_r\}$, and we can select a root $\a_{i_o}\in \Pi$ so that $\Pi\setminus \{\a_{i_o}\}$ is a system of simple roots of the isotropy subalgebra $\gh$. Then it turns out that $\gm_j^{1,0}$ is spanned by the root spaces $\gk_\a$, where $\a$ is a positive root such that $\a=\sum_{i=1}^rn_i\a_i$ and $n_{i_o} = j$ for $j=1,2$.\par
We now claim that $[\gm_1^{1,0},\gm_1^{1,0}]=\gm_2^{1,0}$. Indeed by the description of the modules $\gm_j^{1,0}$, we see that $[\gm_1^{1,0},\gm_1^{1,0}]\subseteq\gm_2^{1,0}$ is
an $\ad(\gh)$-invariant submodule of the irreducible module $\gm_2^{1,0}$ and therefore $[\gm_1^{1,0},\gm_1^{1,0}]=\{0\}$ if our claim does not hold. On the other hand the following inclusions together with their conjugate follow immediately from the structure of $\gm_j^{1,0}$ and its conjugate $\gm_j^{0,1}$:
$$[\gm_1^{1,0},\gm_1^{0,1}]\subseteq \gh^{\mathbb C},\ [\gm_1^{1,0},\gm_2^{1,0}]=\{0\},\ [\gm_1^{1,0},\gm_2^{0,1}]\subseteq \gm_1^{0,1}.$$
Therefore if $[\gm_1^{1,0},\gm_1^{1,0}]=\{0\}$ then $\gm_1^{\mathbb C} + \gh^{\mathbb C}$ is a non trivial ideal of $\gk$ and this contradicts the simplicity of $\gk$.\par
It then follows that there exists two positive roots $\a,\b$ whose root space are contained in $\gm_1^{1,0}$ and such that $\a+\b$ is a root. By Lemma \ref{Lemma1} we have either
$g_{\a+\b}= g_\a+g_\b$ or $g_{\a+\b}=g_\a$ or $g_{\a+\b}=g_\b$. We now observe that the irreducibility of the modules $\gm_j$, $j=1,2$, implies that for every $j=1,2$ and for every  $\g_1,\g_2\in
\gm_j^{1,0}$ we have $g_{\g_1}=g_{\g_2}$. Hence, if $g_{\a+\b}=g_\a+g_\b$ the metric $g$ is K\"ahler, while if either $g_{\a+\b}=g_\a$ or $g_{\a+\b}=g_\b$, then the values
$\{g_\g|\ \g\in R_\gm^+\}$ are all equal and the metric is the restriction of $-B$ up to scaling.\end{proof}
\medskip
For the sake of completeness, in Table \ref{table} we list all flag manifolds with $K$ simple and whose isotropy representation has exactly two irreducible summands (note that when $\gk=\gg_2$ the isotropy subalgebra corresponds to a short root).

\begin{table}[ht]\label{T1}
	\centering
	\renewcommand\arraystretch{1.1}
	\begin{tabular}{|c|c|c|c|}
		\hline
		{\mbox{Type}}				& 	$\gk$					&	$\gh$	& 	{\mbox{dimension}}			 	\\ \hline \hline
	$B$ &	$\so(2\ell +1)$ & $\so(2(\ell-p)+1) + {\mathfrak u}(q) +\mathbb R$ & $p(4\ell+1-3p)$,\, $p\leq \ell$				\\ \hline
	$C$ & ${\mathfrak {sp}}(\ell)$ & ${\mathfrak {sp}}(\ell-p) + {\mathfrak u}(p)$ & $p(4\ell+1-3p)$,\, $p< \ell$			\\ \hline
	$D$ & $\so(2\ell)$ & $\so(2(\ell-p)) + \mathfrak u(p)$ & $p(4\ell-1-3p)$,\, $p< \ell$			\\ \hline
	$E$ & $\mathfrak e_6$ & $\su(5)+\su(2)+\mathbb R$ & $50$			\\ \hline
	$E$ & $\mathfrak e_6$ & $\su(6)+\mathbb R)$ & $42$			\\ \hline
	$E$ & $\mathfrak e_7$ & $\so(10)+\su(2)+\mathbb R$ & $84$			\\ \hline
	$E$ & $\mathfrak e_7$ & $\so(12)+\mathbb R$ & $68$			\\ \hline
	$E$ & $\mathfrak e_7$ & $\su(7)+\mathbb R$ & $84$			\\ \hline
	$E$ & $\mathfrak e_8$ & $\mathfrak e_7+\mathbb R$ & $114$			\\ \hline
	$E$ & $\mathfrak e_8$ & $\so(14)+\mathbb R$ & $156$			\\ \hline
	$F$ & $\mathfrak f_4$ & $\so(7)+\mathbb R$ & $30$			\\ \hline
	$F$ & $\mathfrak f_4$ & $\mathfrak{sp}(3)+\mathbb R$ & $30$			\\ \hline
	$G$ & $\gg_2$ & $\mathfrak u(2)$ & $10$			\\ \hline
	\end{tabular}
	\vspace{0.1cm}
	\caption{Classification of class $\mathcal C$ flag manifolds.\phantom\qquad\qquad \qquad \qquad \qquad }\label{table}
\end{table}

In the following Proposition we will collect some facts about roots in $R_\gm^+$ in case the Lie algebra $\gk$ is supposed to be {\em simply laced}, namely when all roots have the same length. These properties will be used in the proof of Theorem \ref{SU} and may lead to similar results for other flag manifolds.\par
We recall that in a simply laced Lie algebra all roots have the say length, say $c$ and the length of a string of roots is at most $2$. Given two roots $\a,\b$ which are not orthogonal we have $B(\a,\b)=\pm \frac 12 c^2$ and if $\a+\b\in R$, then $\a-\b\not\in R$, so that  so that (1) in Lemma \ref{Lemma1} applies .

We now prove the following
\begin{proposition}\label{simply} Let $F=K/H$ be a flag manifold endowed with an invariant Hermitian metric $g$ which is determined by the positive numbers $\{g_\a\}_{\a\in R_\gm^+}$. We suppose that $\gk$ is simply laced and that $g$ satisfies the {\rm BTP} condition. Let $\a,\b\in R_\gm^+$ be positive roots with $\a+\b\in R$ and $g_{\a+\b}\neq g_\a+g_\b$. Then the following properties holds:
\begin{itemize}
\item[a)] if there is $\g\in R_\gm^+$ such that $\b+\g\in R_\gm^+$ and $\a+\b+\g\in R_\gm^+$, then $$g_\a = g_\b=g_\g=g_{\a+\b}=g_{\b+\g}=g_{\a+\b+\g};$$
\item[b)] if $\a=\lambda+\mu$ for some $\lambda,\mu\in R_\gm^+$, then $g_\lambda= g_\mu=g_\a$.
\end{itemize}
\end{proposition}
\begin{proof}
(a)\ By Lemma \ref{Lemma1}, (1), we know that $g_\a=g_\b=g_{\a+\b}$. Moreover we know that any string of roots has length $2$ and given two roots $\lambda,\mu\in R$ we have that $\lambda+\mu\in R$ if and only if $B(\lambda,\mu) = -\frac 12 c^2$, where $c$ denotes the common length of any root in $R$. Therefore if $\a+\b,\b+\g\in R$ and $\a+\b+\g\in R$, we have that $-\frac 12 c^2=B(\a+\b,\g)=B(\a,\g)+B(\b,\g)=B(\a,\g)-\frac 12 c^2$, hence $B(\a,\g)=0$ and $\a+\g\not\in R$.\par We now consider the condition $\nabla^b_{E_\a}T^b(E_\b,E_\g)=0$ and using the fact that $\a+\g\not\in R$ we have
\begin{eqnarray*}
&&0=\Lambda^b(E_\a)(T(E_\b,E_\g))-T(\Lambda^b(E_\a)E_\b,E_\g) - T(E_\b,\Lambda^b(E_\a)E_\g) \\
&&=\Lb(E_\a)\big(N_{\b,\g}\big( 1-\frac{g_\b+g_\g}{g_{\b+\g}}\big)E_{\b+\g}\big) - N_{\a,\b}\big(1-\frac{g_\a}{g_{\a+\b}}\big)N_{\a+\b,\g}\big(1-\frac{g_\g+g_{\a+\b}}{g_{\a+\b+\g}}\big)\\
&&= N_{\b,\g}N_{\a,\b+\g}\big( 1-\frac{g_\b+g_\g}{g_{\b+\g}}\big)\big(1-\frac{g_\a}{g_{\a+\b+\g}}\big) - N_{\a,\b}N_{\a+\b,\g}\big(1-\frac{g_\a}{g_{\a+\b}}\big)\big(1-\frac{g_\g+g_{\a+\b}}{g_{\a+\b+\g}}\big).
\end{eqnarray*}
As by Jacobi identity $N_{\a,\b}N_{\a+\b,\g}+N_{\b,\g}N_{\b+\g,\a} =0$, we have that
\begin{equation}\label{1} \big( 1-\frac{g_\b+g_\g}{g_{\b+\g}}\big)\big(1-\frac{g_\a}{g_{\a+\b+\g}}\big) = \big(1-\frac{g_\a}{g_{\a+\b}}\big)\big(1-\frac{g_\g+g_{\a+\b}}{g_{\a+\b+\g}}\big).
\end{equation}
The equation $\nabla^b_{E_\b}T^b(E_\a,E_\g)=0$ can be worked out similarly yielding
\begin{equation}\label{2} \big(1-\frac{g_\b}{g_{\a+\b}}\big)\big( 1-\frac{g_\g+g_{\a+\b}}{g_{\a+\b+\g}}\big) = \big(1-\frac{g_\b}{g_{\b+\g}}\big)\big(1-\frac{g_\a+g_{\b+\g}}{g_{\a+\b+\g}}\big).
\end{equation}
Now we claim that $g_{\b+\g}= g_\b = g_\g$. Indeed, otherwise we would have $g_{\b+\g}=g_\b+g_\g$. From \eqref{2} we have $g_{\a+\b+\g}=g_\a+g_{b+\g}=g_\a+g_\b+g_\b$. But
$g_{\a+\b+\g}= g_{(\a+\b)+\g}$ is equal either to $g_{\a+\b}=g_\a$ contradicting $g_\b+g_\g\neq 0$, or to $g_{\a+\b}+g_\g = g_\a+g_\g$, contradicting $g_\b\neq 0$. Therefore from \eqref{1} we have $g_{\a+\b+\g}= g_\a$ and we have proved our claim.
From now on we suppose that $\gk$ {\bf is simply laced}. Suppose the metric is not K\"ahler, i.e. there exist two roots $\a,\b\in R_\gm^+$ so that $\a+\b\in R$ and  $g_{\a+\b}\neq g_\a+g_\b$. It then follows that $g_\a=g_\b=g_{\a+\b}$.\par
(b)\ We consider the roots $\{\lambda,\mu,\beta\}$. Note that $g_{\lambda+\mu+\beta}$ being a root implies that $B(\lambda+\mu,\beta)=-\frac 12 c^2$, so that either $B(\lambda,\beta)=-\frac 12 c^2$ or $B(\mu,\b)=-\frac 12 c^2$, so that we can suppose $\lambda+\beta\in R$.  If we prove that $g_{\lambda+\mu} \neq g_\lambda+g_\mu$, then we can apply (a) and prove our claim.
Assume on the contrary that $g_{\lambda+\mu} = g_\lambda+g_\mu$. Then $g_{\lambda+\mu+\beta}=g_{\a+\b}= g_\a = g_\lambda + g_\mu$. On the other hand $g_{(\lambda+\beta)+\mu}$ is root, hence we have two cases: (1) $g_{(\lambda+\beta)+\mu}=g_\mu$, that is not possible as $g_{(\lambda+\beta)+\mu}= g_\lambda + g_\mu > g_\mu$; (2) $g_{(\lambda+\beta)+\mu}=g_{\lambda+\b}+g_\mu$ and in this case $g_\lambda+g_\mu=g_{(\lambda+\beta)+\mu}=g_{\lambda+\b}+g_\mu$, so that
$g_{\lambda+\b}=g_\lambda$. Therefore $g_{\lambda+\b}=g_\lambda=g_\b$, whence $g_\a=g_\lambda + g_\mu = g_\b + g_\mu$ leading to $g_\mu=0$, a contradiction.
\end{proof}

 Using the previous Proposition, we can now prove Theorem \ref{SU} stated in the Introduction

\begin{proof}[{\bf{Proof of Theorem \ref{SU}.}}] The algebra $\gk= \su(n+1)$ is simply laced and therefore Proposition \ref{simply} applies. We fix the standard maximal abelian subalgebra $\gt$ of $\gk$ together with the system of positive roots given by $R^+ = \{\eps_{i,j}:=\epsilon_i-\epsilon_j| 1\leq i<j\leq n+1\}$. An invariant {\rm BTP} Hermitian metric $g$ on $F$ is determined by the positive numbers $\{g_\a\}_{\a\in R^+}$ and we will indicate $g_{i,j}:=g_{\eps_i-\eps_j}$ for brevity. We suppose that $g$ is not K\"ahler, so that there exists $\a,\b\in R^+$
with $\a+\b\in R^+$ and $g_{\a+\b}\neq g_\a+g_\b$, so that $g_\a=g_\b=g_{\a+\b}$ by Lemma \ref{Lemma1}, (1). We have $\a=\eps_{i,j}$ and $\b=\eps_{j,k}$ for some $1\leq i<j<k\leq n+1$ and moreover we can rescale the metric $g$ so that $g_\a=g_\b=g_{\a+\b}=1$. We claim that $g_\delta=1$ for every positive root $\delta$.\par
We first notice that $\a=\eps_{i,a}+\eps_{a,j}$ with $a=i+1,\ldots,j-1$, and similarly $\b=\eps_{j,b}+\eps_{bk}$ with $b=j+1,\ldots,k-1$, so that using Proposition \ref{simply} (b) repeatedly we obtain that
\begin{equation} \label{start}g_{a,b}= 1\qquad {\rm{if}}\quad i\leq a<b\leq j\quad {\rm {or}}\quad j\leq a<b\leq k.\end{equation}
In particular we have $g_\delta=1$ for every simple root $\delta=\eps_{a,a+1}$ with $a=i,\ldots,k-1$. We claim that $g_\eta=1$ for every simple root $\eta$. Indeed we first observe the following
\begin{lemma}\label{simpleroots} If $\eta_1,\eta_2$ are two simple roots, say $\eta_1=\eps_{a,a+1},\eta_2=\eps_{a+1,a+2}$ with $a=i,\ldots,k-2$, then $g_{\eta_1+\eta_2}=1$.
\end{lemma}
\begin{proof} If $i\leq a<a+2\leq j$ or $j\leq a<a+2\leq k$ the sum $\eta_1+\eta_2=\eps_{a,a+2}$ and $g_{\eta_1+\eta_2}=1$ because of \eqref{start}. Therefore we can suppose
$a=j-1$. We can now distinguish two cases. Either $\a,\b$ are both simple coinciding with $\eta_1,\eta_2$ resp. and in this case $g_{\eta_1+\eta_2}=1$ by hypothesis, or say $\a$ is not simple, hence $j\geq i+2$. This means that we can consider $\eta_3:= \eps_{j-2,j-1}$ with $g_{\eta_3+\eta_1}=1$. By applying Proposition \ref{simply}, (a), we get that
$g_{\eta_1+\eta_2}=1$\end{proof}
In the next Lemma, we will consider simple roots and positive roots of height $2$.
\begin{lemma}\label{height} We have $g_\delta=1$ whenever $\delta$ is a simple root or it is a sum of two simple roots.
\end{lemma}
\begin{proof} Consider the simple root $\eta=\eps_{i-1,i}$, in case $i\geq 2$. Then by Lemma \ref{simpleroots} and Proposition \ref{simply}, (a), applied to the roots $\{\eps_{i-1,i},\eps_{i,i+1},\eps_{i+1,i+2}\}$ we see that $g_{i-1,i}=1$ and also for the sum $\eps_{i-1,i}+\eps_{i,i+1}$ we have $g_{i-1,i+1}=1$. This argument can be repeated
backwards until we reach the first simple root $\eps_{12}$ and at each step we also obtain that the sum of two consecutive simple roots $\eta_1,\eta_2$ has $g_{\eta_1+\eta_2}=1$.
The same procedure works also starting from the simple root $\eps_{k,k+1}$ until we reach the last simple root $\eps_{n,n+1}$.\end{proof}
To conclude the proof of Theorem \ref{SU}, we will show that $g_{\g}=1$ for every $\g\in R^+$. If $\g$ has height less than two, we can apply Lemma \ref{height}. If the height of $\g$ is bigger than $3$, then $\g=\eta_1+\eta_2+\eta$,
where $\eta_1,\eta_2$ are consecutive simple roots and $\eta\in R^+$. By Lemma \ref{height} we have that $g_{\eta_1+\eta_2}=1$ and therefore by Proposition \ref{simply}, (a), we have $g_\g=1$.\end{proof}
\bigskip


\section{The case of a compact Lie group with a Samelson structure}
\vspace{0.1cm}
We consider a compact even-dimensional semisimple Lie group $G$ with a Samelson structure $J$. We recall that a Samelson structure $J$ is a left invariant complex structure on $G$ which is constructed starting from a maximal torus $T\subset G$ and  a system of positive roots $R^+$ relative to the Cartan subalgebra $\gt^{\mathbb C}\subset \gg^{\mathbb C}$ (\cite{Samelson}). The complex structure $J$ acts as $ iI$ (resp. $-iI$) on each positive (resp. negative) root space, while it maps $\gt$ onto itself and the restriction $J|_{\gt}$ can be arbitrarily chosen as a complex structure on $\gt$. A well-known result states that every left-invariant complex structure on $G$ is of this type for a suitable choice of a maximal torus $T$ and therefore it is also invariant under the right action of $T$. This accounts to say that the projection $\pi:G\to G/T$ is holomorphic when we endow the full flag manifold $G/T$ with the induced complex structure and $\pi$ can be seen as the Tits fibration (see e.g.\,\cite{A}).\par
The aim of this section is to prove the result stated in Theorem \ref{Sam}, which provides a characterization of BTP metrics on a compact semisimple Lie group $G$ endowed with a Samelson structure under the additional hypothesis that the metric is invariant under right translations by $T$ or, equivalently, it descends to an invariant Hermitian metric on the base of the Tits fibration.

\begin{proof}[{\bf Proof of Theorem \ref{Sam}}.] The Samelson invariant complex structure is determined by the choice of a maximal torus $T\subset G$ and a set of positive roots relative to the Cartan subalgebra $\gt^{\mathbb C}$. The condition that the metric $g$ descends to a metric on the base $G/T$ of the Tits fibration means that the metric $g$ is also invariant under right translations by elements of $T$. \par
Using now the same arguments used in the computation of the Bismut connection on a flag manifold, we see that
\begin{eqnarray}\label{conn2}
&& \nabla^b_{E_\a}E_\b= N_{\a,\b}\left(1-\frac{g_\a}{g_{\a+\b}}\right) \, E_
{\a+\b},\qquad \mbox{if}\ \a,\b\in R^+,\\
&& \nabla^b_{E_\a}E_\b= 0, \qquad \mbox{if}\ \a\in R^+, \b\in R^-,\a\neq -\b,\quad  \a+\b\not\in R\ {\rm{or}}\ \a+\b\in R^+\\
&& \nabla^b_{E_\a}E_\b= N_{\a,\b}\left(\frac{g_\b-g_\a}{g_{\a+\b}}\right) \, E_
{\a+\b},\qquad \mbox{if}\ \a\in R^{\pm},\b\in R^{\mp}, \a+\b \in R^{\mp}.
\end{eqnarray}
In particular the statement in Lemma \ref{Lemma1} holds true.
We now compute for $\a\in R$ and $H\in \gt^{\mathbb C}$
$$d\o(H,E_\a,E_{-\a}) = -\o([H,E_\a],E_{-\a}) + \o([H,E_{-\a}],E_\a) -\o([E_\a,E_{-\a}],H) =$$
$$=  -\a(H)\o(E_\a,E_{-\a})-\a(H)\o(E_{-\a},E_\a)-\o(H_\a,H)= g(H_\a,JH).$$
Moreover
$$g(U(E_\a,E_{-\a}),H) = g([H,E_\a],E_{-\a}) + g([H,E_{-\a}],E_\a) = -\a(H)g_\a +\a(H)g_\a=0,$$
so that
$$U(E_\a,E_{-\a}) = 0.$$
This implies that
$$g(\nabla^b_{E_\a}E_{-\a},H) = g(\nabla^{\mbox{\tiny LC}}_{E_\a}E_{-\a},H) + \frac 12 d\o(JE_\a,JE_{-\a},JH) = $$
$$ =  \frac 12 g(H_\a,H) + \frac 12 d\o(JE_\a,JE_{-\a},JH) = $$
$$= \frac 12 g(H_\a,H) + \frac 12 d\o(E_\a,E_{-\a},JH) = \frac 12 g(H_\a,H) - \frac 12 g(H_\a,H) = 0$$
so that
\begin{equation}\nabla^b_{E_\a}E_{-\a} = 0,\qquad T^b(E_\a,E_{-\a}) = -H_\a.\end{equation}
From the expression of $U$ we see that $U(E_\a,H) =c_\a E_\a$ for some constant $c_\a$ and therefore
$$g(U(E_\a,H),E_{-\a}) = g([E_{-\a},E_\a],H) + g([E_{-\a},H],E_\a) $$
$$-c_\a g_\a = -g(H_\a,H) + \a(H)g(E_\a,E_{-\a}) = -g(H_\a,H) -\a(H) g_\a$$
hence
$$U(E_\a,H) =(\a(H)+\frac{g(H_\a,H)}{g_\a}) E_\a.$$
This implies that
$$\nabla^{\mbox{\tiny LC}}_HE_\a = \frac 12 \a(H)E_\a + \frac 12 (\a(H)+\frac{g(H_\a,H)}{g_\a}) E_\a = \big(\a(H) + \frac{g(H_\a,H)}{2g_\a}\big) E_\a,$$
$$\nabla^{\mbox{\tiny LC}}_{E_\a}H =  \frac{g(H_\a,H)}{2g_\a} E_\a.$$
It follows that
$$g(\nabla^b_HE_\a,E_{-\a}) = g(\nabla^{\mbox{\tiny LC}}_HE_\a,E_{-\a}) +\frac 12 d\o(JH,JE_\a,JE_{-\a}) =$$
$$=-g_\a\a(H)-\frac{g(H_\a,H)}{2}-\frac{g(H_\a,H)}{2} $$
so that
\begin{equation}\nabla^b_HE_\a = \big(\a(H)+\frac{g(H,H_\a)}{g_\a}\big) E_\a\end{equation}
and similarly,
\begin{equation}\label{Hpar}\nabla^b_{E_\a}H = 0.\end{equation}
Therefore
$$T^b(H,E_\a) = \frac{g(H,H_\a)}{g_\a} E_\a. $$
Finally it is immediate to see that
$$\nabla^b_{H_1}H_2 = 0,\qquad T^b(H_1,H_2)=0,\qquad H_1,H_2\in \gt^{\mathbb C}.$$

\begin{lemma}\label{altern} Let $\a,\b\in R^+$ with $\a+\b\in R^+$. Then $g_{\a+\b}\neq g_\a+g_\b$.
\end{lemma}
\begin{proof} We suppose by contradiction that $g_{\a+\b}= g_\a+g_\b$. Using the expressions for $\nabla^b$ given in \eqref{Hpar} and \eqref{conn}, we compute for every $H\in \gt^{\mathbb C}$
\begin{eqnarray*} 0&=&\nabla^b_{E_\a}T(H,E_\b) = \frac{g(H,H_\b)}{g_\b}\nabla^b_{E_\a}E_\b - T(H,\nabla^b_{E_\a}E_\b)\\
&=& N_{\a,\b}\left(1-\frac{g_\a}{g_{\a+\b}}\right)\cdot\left(\frac{g(H,H_\b)}{g_\b}-\frac{g(H,H_\a+H_\b)}{g_{\a+\b}}\right) E_{\a+\b}\\
&=& N_{\a,\b}\left(1-\frac{g_\a}{g_{\a+\b}}\right)\cdot\frac{g(H,g_\a H_\b-g_\b H_\a)}{g_\b(g_\a+g_\b)}\, E_{\a+\b}.
\end{eqnarray*}
As $\a,\b$ are not proportional, we have $g_\a H_\b\neq g_\b H_\a$ and therefore there exists $H\in \gt^\mathbb C$ so that $g(H,g_\a H_\b-g_\b H_\a)\neq 0$. This implies that $1-\frac{g_\a}{g_{\a+\b}}=0$, hence $g_\a +g_\b = g_{\a+\b}=g_\a$, a contradiction. \end{proof}
We let $\{\a_1,\ldots,\a_r\}$ to be the set of simple roots. As $\a_i-\a_j\not\in R$ for every $i,j=1,\ldots,r$, by Lemma \ref{Lemma1} and Lemma \ref{altern} we have that $g_{\a_i}=g_{\a_j}$ whenever $\a_i+\a_j$ is a root and this happens precisely when the two roots are connected in the Dynkin diagram of $\gg$. As the Dynkin diagram is connected, we see that $g_{\a_i}=g_{\a_j}$ for every $i,j=1,\ldots,r$ and we can normalize this value to be $1$.\par
We now prove that $g_\a=1$ for every $\a\in R^+$. Indeed, we prove our claim on induction on the height of the root $\a$, which is defined as $ht(\a):=\sum_{i=1}^rn_i$, where
$\a=\sum_{i=1}^r n_i\a_i$ for nonnegative integers $n_i$. Given $\a$ with height $q\geq 2$, it is well-known that there exists at least one simple root, say $\a_i$, with $\a-\a_i=\b\in R$. As $ht(\b) = q-1$, by induction hypothesis we have $g_\b=1$. Moreover $\b-\a_i$ is either not a root or a positive root. Using Lemma \ref{altern} and Lemma \ref{Lemma1} we see that $g_\a=g_\b=g_{\a_i}=1$ and this concludes the proof of the first claim.

To prove the second claim, we fix any $J$-Hermitian metric $g_o$ on $\gt$ and extend it to a metric $g$ on $\gg$ by taking $-B$ on each root space. It the follows that the Bismut connection $\nabla^b$ satisfies $\nabla^b_{E_\a}X=0$ for every root $\a$ and every $X\in \gg$, while for $H\in \gt$ we have $\nabla^b_HY=0$ for every $Y\in \gt$ and
$$\nabla^b_{H}E_\a= (\a(H)+ g_o(H,H_\a))E_\a, \qquad \a\in R.$$
To check that $\nabla^b_XT^b=0$ for every $X\in \gg$, we only need to compute: \par
(1) $\nabla^b_HT^b(E_\a,E_\b)$ with $\b\neq -\a$;\par
(2) $\nabla^b_HT^b(E_\a,E_{-\a})$ ;\par
(3) $\nabla^b_HT^b(H',E_\a)$.\par
If we use the fact that $T^b(X,Y) = -[X,Y]$ unless $X=H$ and $Y=E_\a$ with $T^b(H,E_\a) = g(H,H_\a)E_\a$, it is straightforward to verify that $\nabla^bT^b=0$ in the cases (1)-(3) and this concludes the proof.\end{proof}
\begin{remark}\label{rmkSam} The proof actually shows that the {\rm BTP} property does not depend on the restriction of the metric along the fibers of the Tits fibration. Moreover we easily check that the only possibly non-vanishing component of the curvature tensor $R^b$ is given by $R^b(E_\a,E_{-\a})E_\b=-[\b(H_\a)+g(H_\a,H_\b)]E_\b$ for roots $\a,\b$. From this one concludes that $\nabla^bR^b=0$ and the connection $\nabla^b$ is Ambrose-Singer. \end{remark}
\vspace{.3cm}
\section{A conjecture on homogeneous {\rm BTP} metrics}

In this section we focus on {\rm BAS} metrics on compact Hermitian manifolds. According to Definition \ref{defAS},  we recall that a Hermitian metric is called {\rm BAS} when its Bismut connection has parallel torsion and curvature. By the aforementioned result of Ambrose-Singer and Sekigawa, we know that a necessary condition for a {\em {\rm BTP}} metric to be {\rm BAS} is that the metric is locally homogeneous, namely, its universal covering space is a homogeneous Hermitian manifold. Another necessary condition for {\rm BAS} is that the metric needs to be free of  non-locally symmetric K\"ahler de Rham factors, as K\"ahler metrics are always {\rm BTP}, but they are {\rm BAS} only when they are locally Hermitian symmetric. It is then natural to ask the following:
\begin{question} \label{question5.1}
Let $(M^n,g)$ be a compact Hermitian manifold that is {\rm BTP} with its universal cover being a homogeneous Hermitian manifold without K\"ahler de Rham factor. When will $g$ be {\rm BAS}?
\end{question}

The answer is positive when $n=2$ as observed in \cite{NiZheng}, namely, any non-K\"ahler locally Hermitian homogeneous {\rm BTP} surface is always {\rm BAS}. This is because {\rm BTP} surfaces are exactly Vaisman surfaces by \cite[Theorem 2]{ZhaoZ19Str}, and such surfaces (if not K\"ahler) always admit local unitary frames $e$ under which the only possibly non-zero component of the Bismut curvature is $R^b_{1\bar{1}1\bar{1}}$, which equals the Bismut scalar curvature $s^b$, so $\nabla^bR^b=0$ if and only if $s^b$ is constant. When the surface is locally homogeneous, $s^b$ is constant thus it is {\rm BAS}. The converse is certainly true as well. So one could state that:

{\em Given any non-K\"ahler {\rm BTP} surface, it is {\rm BAS} $\Longleftrightarrow$  it has constant scalar curvature $\Longleftrightarrow$ it is locally Hermitian homogeneous.}

Note that here we did not specify which scalar curvature that was used, as on any {\rm BTP} surface the scalar curvature of Levi-Civita or Chern or Bismut connection, or the trace of the first or second Ricci curvature of Chern or Bismut connection, all differ by constants. So any one of them equals to constant means all of them are constants.

 We now give some examples where the answer to Question \ref{question5.1} is positive.

Our first example concerns Hopf manifolds. We remark that in \cite{AV} it is proved that a locally conformally K\"ahler (LCK) manifold is {\rm BTP} if and only if it is Vaisman, so that diagonal Hopf manifolds are the only linear Hopf manifolds (see e.g. Cor. 16.16 in \cite{OV}) that can carry a {\rm BTP} metric.
\begin{example}[Isosceles Hopf manifolds] Consider the Hopf manifold $M^n=({\mathbb C}^n\setminus \{ 0\})/ \Gamma $ with the deck transformation group $\Gamma \cong {\mathbb Z}$ generated by the contraction $f$ from ${\mathbb C}^n\setminus \{ 0\}$ onto itself given by
$$f(z_1, \ldots , z_n) = ( a_1z_1, \ldots , a_nz_n),  $$
where $a_i\in\mathbb C$ satisfy $0<|a_1|=\cdots =|a_n|<1$. We will call such Hopf manifolds {\em isosceles  Hopf manifolds}. In this case, the Hermitian metric $g$ with K\"ahler form
$$ \omega = \sqrt{-1} \, \frac{ \partial \overline{\partial} |z|^2}{|z|^2}  $$
descends to $M^n$. Here and below we write $|z|^2=|z_1|^2 + \cdots + |z_n|^2$. As  is well-known, $(M^n,g)$ is Vaisman, hence by \cite{AV} it is BTP. We claim that it is BAS.

To verify this, let us write $e_i=|z|\frac{\partial}{\partial z_i}$, $\varphi_i = \frac{1}{|z|}dz_i$ for each $1\leq i \leq n$. Then $e$ forms a local unitary frame on $(M^n,g)$ with dual coframe $\varphi$. Under the frame $e$, a straight forward computation leads us to the formula for the matrix of Bismut connection and components of the Chern curvature
\begin{eqnarray*}
&& \theta^b \ = \ \frac{1}{2}(\partial - \overline{\partial}) \log |z|^2 I + \frac{1}{|z|^2} \big( d\overline{z} \,^t\!z - \overline{z} \,d\,^t\!z \big), \\
&&  R_{i\bar{j}k\bar{\ell}} \ = \ \xi_{ij} \delta_{k\ell}, \ \ \ \ \ \ \ \ \ \ \xi \ = \ I - \frac{1}{|z|^2} \overline{z}\,^t\!z.
\end{eqnarray*}
We compute the matrix commutator
\begin{eqnarray*}
 [\theta^b, \xi] & = & [ \, \frac{1}{|z|^2} \big( d\overline{z} \,^t\!z - \overline{z} \,d\,^t\!z \big) , \, - \frac{1}{|z|^2} \overline{z}\,^t\!z \, ] \\
 & = & \frac{1}{|z|^4} \{ -d\overline{z} \,|z|^2 \,^t\!z + \overline{z} \,\partial |z|^2 \,^t\!z + \overline{z}\, \overline{\partial} |z|^2 \,^t\!z - \overline{z}\, |z|^2 d\,^t\!z \} \\
 & = & \frac{ d|z|^2}{|z|^4} \,\overline{z} \,^t\!z - \frac{1}{|z|^2} \,d(\overline{z} \,^t\!z) \ = \ d\xi
\end{eqnarray*}
Therefore,
\begin{eqnarray*}
 dR_{i\bar{j}k\bar{\ell}}  & = & d\xi_{ij} \,\delta_{k\ell} \ = \ \big( \theta^b \xi - \xi \theta^b)_{ij} \,\delta_{k\ell} \\
 & = & \sum_r \{  R_{r\bar{j}k\bar{\ell}} \,\theta^b_{ir} -  R_{i\bar{r}k\bar{\ell}} \,\theta^b_{rj} \}  \\
 & = & \sum_r \{  R_{r\bar{j}k\bar{\ell}} \,\theta^b_{ir} -  R_{i\bar{r}k\bar{\ell}} \,\theta^b_{rj} + R_{i\bar{j}r\bar{\ell}} \,\theta^b_{kr} -  R_{i\bar{j}k\bar{r}} \,\theta^b_{r\ell } \}
\end{eqnarray*}
where the last equality is due to the fact that $[\theta^b, I]=0$. This means that $\nabla^bR=0$ which is equivalent to $\nabla^bR^b=0$ for BTP manifold, hence $(M^n,g)$ is BAS.

\end{example}
\begin{example}[Calabi-Eckmann manifolds] We consider the homogeneous space $M=S^{2m_1+1}\times S^{2m_2+1}= G/H$, where $G = G_1\times G_2$, $H = H_1\times H_2$ and
$G_i = \SU(m_i+1)$, $H_i= \SU(m_i)$ for $i=1,2$. We also fix a reductive decomposition
$$\mathfrak g_i = \gh_i + \gz_i +\gm_i,\quad i=1,2, $$
where $\gz_i$ is the one-dimensional centralizer of $\gh_i$ in $\frak g_i$ and $\gm_i$ is an $\rm{ad}(\gh_i+\gz_i)$-invariant subspace. Note that $\gh_i+\gz_i$ is the Lie algebra of the subgroup $K_i:= \rm{S}(\rm U(1)\times \rm U(m_i))$ and that $G_i/K_i \cong \mathbb CP^{m_i}$. The subspace $\gm_i$ is isomorphic to the tangent space $T_o\mathbb CP^{m_i}$ at the origin $o=eH_i$ and it is endowed with a ${\rm Ad}(K_i)$-invariant complex structure $J_i$.

We can endow $M$ with an invariant complex structure $J$ as follows (see e.g.\,\cite{Tsukada},\,\cite{B12}). We fix generators $z_i$ of $\gz_i$ and we put $\gq:= {\rm{Span}}(z_1,z_2)$. We then define $J$ so that $J(\gm_i)=\gm_i$ with $J|_{\gm_i}=J_i$ and $J|_{\gq}$ is the endomorphism which can be expressed  by the matrix $\left(\begin{smallmatrix} \alpha & -\frac{1+\alpha^2}\beta \\ \beta& -\alpha\end{smallmatrix}\right)$ for some $\alpha,\beta\in \mathbb R,\beta\neq 0$, with respect to the basis $\{z_1,z_2\}$. The complex structure $J$ is easily seen to define an invariant integrable complex structure on $M$.

Moreover, we can define $J$-hermitian invariant metrics $g$ on $M$ in the following way. We decree
$$g(\gm_i,\gm_j)=0\ {\rm{if}}\, i\neq j,\qquad g(\gm_i,\gq)=0,\, i=1,2,$$
$$g|_{\gm_i\times \gm_i} = -c_iB_i|_{\gm_i\times \gm_i},\, i=1,2,$$
where $c_i<0$ are negative real numbers and $B_i$ denotes the Cartan-Killing form of $\frak g_i$. We then define $g$ on the subspace $\gq$ to be any $J$-Hermitian metric: any such $g$ is a positive multiple of the metric whose matrix is $\left(\begin{smallmatrix} 1 & -\frac\alpha \beta \\ -\frac\alpha\beta& \frac{1+\alpha^2}{\beta^2}\end{smallmatrix}\right)$ w.r.t. the basis $\{z_1,z_2\}$.

We will show that all these homogeneous spaces $(M,g,J)$, which are called Calabi-Eckmann manifolds, are naturally reductive, so that the corresponding Bismut connections have parallel curvature, hence $g$ is ({\rm BAS}).

In order to do this, we will enlarge the group $G=G_1\times G_2$ and observe that the isometry group of $(M,g)$ contains the group $\hat G:= G\times A$, where $A$ is a compact $2$-dimensional torus acting on $M$ as follows. We denote by $Q$ the $2$-torus with Lie algebra $\gq$ and we fix an isomorphism $\rho:A\to Q$; then an element $a\in A$ acts on $M$ as $a\cdot xH := x\rho(a)H$. This action is holomorphically isometric as both $J$ and $g$ are ${\rm{Ad}}(Q)$-invariant. The isotropy subgroup $\hat H$ at $o$ is given by
$$\hat H = \{(g,q)\in \hat G|\ g\rho(a)\in H\} = \{(h\rho(a)^{-1},a)|\ a\in A,h\in H\}.$$
We consider a reductive ${\rm{ad}}(\hat \gh)$-invariant decomposition
\begin{equation}\label{reddec}\hat \gg = \hat \gh + V_f + \hat\gm\end{equation}
where $\hat\gm = \{(x,0)\in \hat\gg|\ x\in \gm_1+\gm_2\}$ and
$$V_f=\{(q,f(q))\in \gq+\ga|\, q\in \gq\}$$
for some linear $f:\gq \to \ga$ such that $f\rho_*\in {\rm{End}}(\ga)$ does not have $\lambda = -1$ as an eigenvalue (hence $V_f\cap \hat \gh = \{0\}$).

It is not difficult to show that we can choose a suitable $f$ so that the metric $g$ expressed in the reductive decomposition \eqref{reddec} satisfies the naturally reductive condition and we leave the details to the reader.
\end{example}
\begin{example}[Vaisman solvmanifolds] If $M =G/\Gamma$ is a solvmanifold, where $G$ is a solvable group and $\Gamma$ is a cocompact lattice in $G$, then a Vaisman LCK-structure $(g,J)$ induced by a left-invariant LCK-structure on $G$ is {\rm BAS}, as it has been shown in \cite{AV}.
\end{example}
\begin{example} In the previous sections, we have provided other significant examples, which can be summarized as follows:
\begin{itemize}
\item[i)] Any Chern flat Hermitian metric on a compact complex manifold is {\rm BAS} whenever it is {\rm BTP} (see Proposition \ref{ChernflatBAS});
\item[ii)] any invariant {\rm BTP} metric on a generalized flag manifold with at most two irreducible summands or on the full flag manifolds $\SU(n+1)/\T^n$ is either K\"ahler or BAS (see Proposition \ref{flag2summ} and Theorem \ref{SU});
    \item[iii)] given a compact Lie group $G$ endowed with a Samelson complex structure w.r.t. a maximal torus $\T$, any left invariant {\rm BTP} metric on $G$ which is invariant by right $\T$-translations is {\rm BAS} (see Theorem \ref{Sam}, Remark \ref{rmkSam}).
\end{itemize}
\end{example}
At this point in time we do not know whether or not the answer to Question \ref{question5.1} is always yes when $n=3$. However,  we are able to exhibit an example of
a compact homogeneous complex $4$-fold $M^4$ which is endowed with a {\rm BTP} metric that is not {\rm BAS}.  The manifold $M^4$ is defined as the product $N\times S^1$, where
$N$ is a Sasaki homogeneous manifold fibering over the Wallach $3$-fold $W^3:= \SU(3)/\T$, where $\T\subset \SU(3)$ is the maximal torus given by all diagonal matrices in $\SU(3)$. Then $M^4$ is Vaisman and therefore {\rm BTP}. For the sake of clarity, we will carry out the construction in a detailed way.  \par
Namely, we consider the Wallach $3$-fold $W^3:= \SU(3)/\T$ and we will denote by $R$ the root system corresponding to $\T$ and with the standard choice of simple roots $\{\a:=\epsilon_{12},\beta:=\epsilon_{23}\}$. This choice determines the Weyl chamber $\mathcal C=\{v\in\gt \,|\,i\a(v)>0,i\b(v)>0\}$, which in turn determines an invariant complex structure $J_o$ on $W^3$ : in the complexification $\gm^{\mathbb C}$ the subspace $\gm^{1,0}$ is spanned by the root spaces corresponding to positive roots.

We now select an element $\ell\in\mathcal C$ that generates the Lie algebra $\gl$ of a one-dimensional torus $\rm L$. If we fix the ${\rm{Ad}}(\T)$-invariant decomposition ${\frak {su}}(3)=\gt +\gm$ with $\gm$ being the $B$-orthogonal complement of $\gt$, it is well-known (see e.g.\,\cite{BFR}) that $\ell\in \mathcal C$ determines an invariant K\"ahler metric $g_o$ on $W$ whose K\"ahler form $\omega_o$ is given by
\begin{equation} \omega_o(x,y) = B([\ell,x],y),\quad x,y\in\gm,\end{equation}
where $B$ denotes the Cartan Killing form of $\frak{su}(3)$. For simplicity we may fix $\ell= i\,{\rm{diag}}(a_1,a_2,-a_1-a_2)$ with $a_1,a_2\in\mathbb Z\backslash\{0\}$ satisfying
$a_1<a_2<-\frac 12 a_1$.

The subspace $\gh$ of $\gt$ which is $B$-orthogonal to $\gl$ is generated by $h:=i\,{\rm{diag}}(2a_2+a_1,-2a_1-a_2,a_1-a_2)$ and it is the Lie algebra of a compact one-dimensional torus $\rm H$. The fibration $\SU(3)/\rm H\to W^3$ is a Boothby-Wang fibration and we define $M^4 := (\SU(3)/\rm H) \times S^1$.

The homogeneous manifold $M^4$ can be endowed with an invariant complex structure as follows. If $\gz$ denotes the Lie algebra of the $\rm S^1$-factor, we have an ${\rm{Ad}}(\rm H)$-invariant decomposition
$$\frak{su}(3)+\gz = \gh + \gl + \gm +\gz $$
where the subspace $\gp:= \gl + \gm +\gz $ identifies with the tangent space $TM^4$ at the origin. We can then define an invariant almost complex structure $J$ by taking $J_o$ on $\gm$ and by decreeing $J(\gz)=\gl$. It is immediate to verify that $J$ defines an integrable complex structure on $M^4$. Moreover we can define an invariant $J$-Hermitian metric $g$ by taking $g|_{\gm\times\gm}= g_o$, $g(\gm,\gl+\gz)=0$ and $g(\ell,\ell) = -B(\ell,\ell)$.  Note that the product $(\SU(3)/\rm H) \times S^1$ is an
isometric splitting and $\rm S^1$ acts by holomorphic isometries.

\begin{proposition} The homogeneous Hermitian manifold $(M^4,J,g)$ is {\rm BTP} but not {\rm BAS}.
\end{proposition}
\begin{proof} First of all we will show that $(M^4,J,g)$ is LCK and in particular Vaisman, so that $g$ is {\rm BTP} by the result in \cite{AV}.
If $\omega$ denotes the K\"ahler form of $g$, we compute $d\omega$ in the following cases.

(1)\ When $x,y,z\in\gm$, then
$$d\omega(x,y,z)= -\omega([x,y]_\gp,z) + \omega([x,z]_\gp,y) + \omega([y,z]_\gp,x) = 0$$
as $\gm$ is $J$-stable and $g_o$ is K\"ahler.

(2)\ When $x,y\in\gm$, then
\begin{eqnarray*}d\omega(x,y,\ell) &=& -\omega([x,y]_\gp,\ell)+ \omega([x,\ell],y) -\omega([y,\ell],x)\\
&=& g([x,y]_\gp,J\ell) + \omega([x,\ell],y) +\omega(x,[y,\ell])=0
\end{eqnarray*}
as $[x,y]_\gp\in (\gl+\gm) \perp \gz$ and $\omega|_{\gm\times\gm}$ is $\rm{ad}(\gt)$-invariant by construction.

(3)\ When $x,y\in \gm$, then as $[J\ell,\frak{su}(3)]=0$ we have
\begin{eqnarray*}d\omega(x,y,J\ell) &=& -\omega([x,y]_\gp,J\ell)+ \omega([x,J\ell],y) -\omega([y,J\ell],x)\\
&=& -g([x,y]_\gp,\ell) = -g([x,y]_\gl,\ell) = B([x,y],\ell) = B([\ell,x],y) = \omega(x,y).
\end{eqnarray*}
(4)\ When $x\in\gm$, then
$$d\omega(x,\ell,J\ell) = -\omega([x,\ell]_\gp,J\ell)+ \omega([x,J\ell],\ell) -\omega([\ell,J\ell],x) = 0.$$
Therefore if we set $\theta = (J\ell)^\flat$, we obtain $d\omega=\theta\wedge\omega$.  As $J\ell\in\gz$ and the splitting $(\SU(3)/\rm H) \times S^1$ is isometric, we see that $\theta$ is parallel w.r.t. the Levi Civita connection of $g$ and therefore the manifold is LCK, in particular Vaisman.

We will now show that the metric is not {\rm BAS} by proving that the homogeneous space $M^4$ is not naturally reductive. To this end, we first characterize the connected component of the full group $\rm G$ of holomorphic isometric transformations of $M^4$. The group $\rm G$ is clearly compact.
\begin{lemma} The group $\rm G$ is locally isomorphic to $\SU(3) \times \rm A \times {\rm S}^1$, where $\rm A$ is a one-dimensional torus acting on $M^4$ by right translations.
\end{lemma}
\begin{proof} We fix an isomorphism $\rho:\rm A\to \rm L$  and we let $A$ act on $M^4$ by means of right translations by corresponding elements of $\rm L$. This action is holomorphically isometric because the complex structure $J$ and the metric $g$ are both $\ad(\gl)$-invariant. Therefore the compact group $\rm G$ contains a subgroup $\rm U$ locally isomorphic to  $\SU(3) \times A \times {\rm S}^1$. On the other hand the projection $\pi:M^4\to W^3$ is holomorphic and by Blanchard theorem (see \cite{A}, p.45) every element of $\rm G$ preserves the fibers of the fibration $\pi$ and descends to a holomorphic transformation of $W^3$. It is known that the group $B$ of biholomorphisms of $W^3$ is locally isomorphic to $\rm{SL}(3,\mathbb C)$ (see \cite{A}) and therefore the induced homomorphism $\phi:\rm G\to B$ maps $\rm G$ onto a compact subgroup of $B$ acting transitively on $W^3$, hence locally isomorphic to $\SU(3)$. If $N:=\ker \phi$, then
$\rm G$ is locally isomorphic to $\SU(3)\cdot N$. The compact group $N$ maps each fiber of $\pi$ onto itself and the induced action on each of them is faithful because $N$ acts by isometries and it acts trivially on the base $W^3$. Since the fibers are tori, we see that $\dim N\leq 2$  and therefore $\dim \rm G \leq \dim \SU(3) + 2 = \dim \rm U$, proving our claim.\end{proof}
We now prove that $M^4$ is not naturally reductive. Indeed, let $\rm G'$ be a connected subgroup of $\rm G$ acting transitively on $M^4\cong\rm G'/\rm H'$ for some subgroup $\rm H'\subseteq \rm G'$. The closure $\overline{\rm G'}$ in $\rm G$ is compact and the same argument as in the proof of the previous Lemma shows that $\overline{\rm G'}$ is locally isomorphic to $\SU(3)\cdot N'$ for some abelian group $N'$. As $\rm G'$ is normal in the compact $\overline{\rm G'}$, we conclude that $\SU(3)\subseteq\rm G'$, hence $\rm H\subseteq \rm H'$. As $\SU(3)$ is not transitive on $M^4$, we have that $\dim \rm G'\geq 9$, hence either $\dim \frak g' = \dim\frak g - 1$ or $\rm G=\rm G'$.
Now suppose we have  a reductive decomposition $\frak g' = \gh' +\gp'$ and that the metric $g$ on $\gp'$ satisfies the natural reductiveness condition for all $x,y,z\in\gp'$
\begin{equation}\label{red} g([x,y]_{\gp'},z) = - g(y,[x,z]_{\gp'}).\end{equation}
When $\dim \frak g' = \dim\frak g - 1$ we have $\gh=\gh'\subset \mathfrak {su}(3)$ by a dimension count. The $\gh$-module $\gm$ is contained in $\frak{su}(3)\subset \frak g'$ and moreover it does not contain any trivial $\gh$-module as the element $h$ is not annihilated by any root. This implies that $\gm\subset \gp'$.
But then equation \eqref{red} does not hold for suitable $x,y,z\in\gm$. Indeed, we consider the complexification of $\gm$ and root vectors $E_\a,E_\b,E_{-\a-\b}\in \gm^\mathbb C$. Then \eqref{red} would give
$$0 = g([E_\a,E_\b],E_{-\a-\b}) + g(E_\b,[E_\a,E_{-\a-\b}]) = N_{\a,\b}\left ( g(E_{\a+\b},E_{-\a-\b}) - g(E_\b,E_{-\b})\right),$$
while for K\"ahler metrics it is well known that $g(E_{\a+\b},E_{-\a-\b})= g(E_\a,E_{-\a})  + g(E_\b,E_{-\b})\neq g(E_\b,E_{-\b})$, leading to a contradiction.

When $\frak g' = \frak g$, we have that $\gh' = \{(y-\rho_*(x),x,0)\in \frak{su}(3)\oplus\ga\oplus\gz|\ x\in \ga, y\in\gh\}$, where $\rho_*$ denotes the differential of $\rho$ at the identity, hence $\gm\subset\frak{su}(3)\subset \frak g'$ is a $\gh'$-module. As $\gh\subset\gh'$, the $\gh'$-module $\gm$ does not contain non-trivial submodules and therefore $\gm\subset \gp'$. As in the previous case this leads to a contradiction. \end{proof}
Note that  the manifold $M^4$ we have constructed has infinite fundamental group and  we would like to propose the following conjecture which has already been stated in the Introduction as Conjecture \ref{conj1.6}:

\begin{conjecture} \label{conj5.2}
Let $(M^n,g)$ be a compact Hermitian manifold whose universal cover is a homogeneous Hermitian manifold without K\"ahler de Rham factor. Assume $g$ is {\rm BTP}, and the manifold is one of the following
\begin{enumerate}
\item $M^n$ has finite fundamental group, or
\item the universal cover of $(M^n,g)$ is a Lie group equipped with a left-invariant metric and a compatible left-invariant complex structure.
\end{enumerate}
Then $g$ must be {\rm BAS}.
\end{conjecture}

As partial supporting evidence of Conjecture \ref{conj5.2}, we prove the following two special cases. The first one states that Case (1) of Conjecture \ref{conj5.2} holds true if $n=3$, while the second one says that Case (2) of Conjecture \ref{conj5.2} is true for nilmanifolds with nilpotent $J$ (in the sense of Cordero, Fern\'{a}ndez, Gray, and Ugarte \cite{CFGU}). These two results were stated in Introduction as Propositions \ref{prop5.3} and \ref{prop5.5}.  We start with the proof of the first one.

\begin{proof}[{\bf Proof of Proposition \ref{prop5.3}}]
The fundamental group of $M^3$ is assumed to be finite. Replacing $M^3$ by its universal cover if necessary, we may assume that $M^3$ is simply-connected. The metric $g$ is homogeneous and {\rm BTP} and without K\"ahler de Rham factor. Let us divide the discussion into two cases depending on whether $g$ is balanced or not.

{\em Case 1.} $g$ is not balanced.

By the work of \cite{ZhaoZ22}, we know that any non-balanced {\rm BTP} manifold $(M^n,g)$ always admits {\em admissible frames}, which are local unitary frames $e$ so that
$$ \eta_1= \cdots = \eta_{n-1}=0, \ \ \eta_n=\lambda >0, \ \ \ T^j_{in}=a_i\delta_{ij} , \ \ \ \forall \ 1\leq i,j\leq n,$$
where $\lambda >0$, $a_1, \ldots , a_{n-1}$ are global constants on $M^n$, with $a_n=0$ and $a_1+ \cdots + a_{n-1}=\lambda$.  Here as before we denoted by $T^j_{ik}$ the components under $e$ of the Chern torsion, and $\eta = \sum_i \eta_i\varphi_i$ is Gauduchon's torsion $1$-form defined by $\partial \omega^{n-1} = -\eta \wedge \omega^{n-1}$, where $\varphi$ is the coframe dual to $e$, and $\omega$ is the K\"ahler form of $g$. One has $\eta_i = \sum_k T^k_{ki}$. Denote by $\theta^b$, $\Theta^b$ the matrix of connection and curvature under $e$ for the Bismut connection $\nabla^b$, and denote by $R^b$ the curvature of $\nabla^b$. Since $g$ is {\rm BTP}, one has $\nabla^b e_n=0$, hence $\theta^b_{n\ast}=0$ and $\Theta^b_{n\ast}=0$.

Now let us specify to our case $n=3$. First consider the case when $a_1\neq a_2$. We claim that we must have $\theta^b_{12}=0$ in this case. To see this, we compute from $\nabla^bT=0$ that
$$ 0 = (\nabla^bT)^2_{13} = d (T^2_{13}) - \sum_r \big( T^2_{r3}\theta^b_{1r} +T^2_{1r}\theta^b_{3r} - T^r_{13}\theta^b_{r2} \big) = (a_1-a_2)\theta^b_{12}.$$
Therefore $\theta^b$, hence $\Theta^b$, is diagonal. The only possibly non-zero components of $R^b_{i\bar{j}k\bar{\ell}} = \Theta^b_{k\ell }(e_i, \overline{e}_j)$ are $R^b_{1\bar{1}1\bar{1}}$, $R^b_{1\bar{1}2\bar{2}}= R^b_{2\bar{2}1\bar{1}}$, and $R^b_{2\bar{2}2\bar{2}}$. Here we used the symmetry property $R^b_{i\bar{j}k\bar{\ell}} =R^b_{k\bar{\ell}i\bar{j}}$ satisfied by any {\rm BTP} metric (see (1.2) of Theorem 1.1 in \cite{ZhaoZ22}). Since $\theta^b$ is diagonal, we have
\begin{eqnarray*} (\nabla^bR^b)_{i\bar{j}k\bar{\ell}} & = &  d(R^b_{i\bar{j}k\bar{\ell}}) - \sum_r \big( R^b_{r\bar{j}k\bar{\ell}} \theta^b_{ir} - R^b_{i\bar{r}k\bar{\ell}} \theta^b_{rj} + R^b_{i\bar{j}r\bar{\ell}} \theta^b_{kr} - R^b_{i\bar{j}k\bar{r}} \theta^b_{r\ell } \big) \\
& = &  d(R^b_{i\bar{j}k\bar{\ell}}) - R^b_{i\bar{j}k\bar{\ell}} \big( \theta^b_{ii} -  \theta^b_{jj} + \theta^b_{kk} -  \theta^b_{\ell \ell } \big)
\end{eqnarray*}
When $i\neq j$ or $k\neq \ell$, then $R^b_{i\bar{j}k\bar{\ell}}=0$ hence the above quantity is $0$. When $i=j$  and $k=\ell$, it equals to $d(R^b_{i\bar{i}k\bar{k}})$. So we will have the desired condition $\nabla^bR^b=0$ provided that we know that the functions $x=R^b_{1\bar{1}1\bar{1}}$, $y=R^b_{1\bar{1}2\bar{2}}$, $z=R^b_{2\bar{2}2\bar{2}}$ are all constants. The first Bismut Ricci form
$$ \rho^{b(1)}= \sqrt{-1} \mbox{tr}(\Theta^b) = \sqrt{-1} (\Theta^b_{11} + \Theta^b_{22}) = (x+y)\sqrt{-1} \varphi_1\overline{\varphi}_1 + (y+z) \sqrt{-1}\varphi_2\overline{\varphi}_2 $$
is a global $(1,1)$-form on $M^3$. Its eigenvalues with respect to the K\"ahler form $\omega$ are $x+y, y+z, 0$. Since any holomorphic isometry of $M^3$ will preserve $\rho^{b(1)}$ and $\omega$, the homogeneity of $M^3$ implies that $x+y$ and $y+z$ are both constants. Similarly, we may consider the third Bismut Ricci, which is the usual (Riemannian) Ricci curvature of the  metric connection $\nabla^b$, defined by
$$ \mbox{Ric}^{b(3)}_{i\bar{j}} = \sum_r R^b_{ r\bar{j} i\bar{r} } $$
under any unitary frame. In our case, we have
$$ \rho^{b(3)}= \sqrt{-1} \sum_{i,j} \mbox{Ric}^{b(3)}_{i\bar{j}} \varphi_i \overline{\varphi}_j = x\sqrt{-1} \varphi_1\overline{\varphi}_1 + z\sqrt{-1}\varphi_2\overline{\varphi}_2 .$$
Since $\rho^{b(3)}$ is preserved by any holomorphic isometry, by considering $ \rho^{b(3)}\wedge \omega^2 $ and $(\rho^{b(3)} )^2\wedge \omega $, we know that $x$ and $z$ are constants, hence $y$ is also a constant. This establishes $\nabla^bR^b=0$ under the assumption that $a_1\neq a_2$.

When $a_1=a_2$, by Case (3) in Theorem 1.5 of \cite{ZhaoZ22}, we know that $(M^3,g)$ must be a Vaisman manifold. But we assumed $M^3$ to be simply-connected here, so we must have $g$ being K\"ahler, contradicting with the assumption that it is without K\"ahler de Rham factor. This concludes the proof in the non-balanced case.

{\em Case 2.} $g$ is balanced.

Assume that $(M^3,g)$ is a compact, balanced, non-K\"ahler {\rm BTP} manifold. Such threefolds were classified by Theorem 1.7 in \cite{ZhaoZ22}. Depending on the rank $r_B$ of the $B$-tensor, which is defined by $B_{i\bar{j}} =\sum_{r,s} T^j_{rs} \overline{T^i_{rs} } $ under any unitary frame, there are three possibilities.

When $r_B=1$, $(M^3,g)$ is the Wallach threefold, which is the flag threefold ${\mathbb P}(T_{{\mathbb P}^2})$ equipped with the Cartan-Killing metric, and the metric is {\rm BAS}.

When $r_B=3$, $(M^3,g)$ is a compact quotient of the complex Lie group $SO(3,{\mathbb C})$ equipped with a left-invariant metric. In this case the metric is Chern flat and by  Proposition \ref{ChernflatBAS}, Chern flat {\rm BTP} metrics are always {\rm BAS}.

When $r_B=2$, we are in the so-called {\em middle type}. Denote by $J$ the complex structure of $M^3$. According to Case (2) of Theorem 1.7 in \cite{ZhaoZ22}, either $M$ or a double cover of $M$ (which for simplicity we will still denote as $M$) will admit another complex structure $I$ compatible with the metric $g$, so that the resulting Hermitian threefold $(M,g,I)$ is Vaisman and non-K\"ahler. This will force the fundamental group of $M$ to be infinite, so this case can not occur either. Analogous to the $r_B=3$ (Chern flat) case, one can directly verify that $g$ is {\rm BAS} here without assuming $\pi_1(M)$ to be finite. Also, since $(M,g,J)$ and $(M,g,I)$ share the same Bismut connection, the latter would also be {\rm BAS}.

Assume that $(M^3,g)$ is a compact, balanced, non-K\"ahler {\rm BTP} threefold with $r_B=2$. From the discussion of \S 9 in \cite{ZhaoZ22}, we know that locally there always exist special unitary frames $e$ under which the connection and curvature matrix of $\nabla^b$ are given respectively by
$$ \theta^b = \left[ \begin{array}{ccc} \alpha & \beta_0 & 0 \\ - \beta_0 & \alpha & 0 \\ 0 & 0 & 0 \end{array} \right] , \ \ \ \  \Theta^b = \left[ \begin{array}{ccc} d\alpha & d\beta_0 & 0 \\ - d\beta_0 & d\alpha & 0 \\ 0 & 0 & 0 \end{array} \right] , $$
where $\alpha$, $\beta_0$ are $1$-forms satisfying $\overline{\alpha}=-\alpha$, $\overline{\beta}_0=\beta_0$. Furthermore, by formula (9.9) and (9.10) in \cite{ZhaoZ22}, one has
\begin{eqnarray*}
&& \Theta^b_{11} = \Theta^b_{22} = d\alpha = x (\varphi_{1\bar{1}} + \varphi_{2\bar{2}}) + \sqrt{-1}\,y (\varphi_{2\bar{1}} - \varphi_{1\bar{2}}) \\
&& \Theta^b_{12} = - \Theta^b_{21} = d\beta_0 = -\sqrt{-1}\,y (\varphi_{1\bar{1}} + \varphi_{2\bar{2}}) + (x+2) (\varphi_{2\bar{1}} - \varphi_{1\bar{2}})
\end{eqnarray*}
where $x$, $y$ are real-valued functions and $\varphi_{i\bar{j}} = \varphi_i \wedge \overline{\varphi}_j$. Here we have scaled the metric so that $|T|^2=1$.  The first Bismut Ricci form $\rho^{b(1)}=\sqrt{-1}\,\mbox{tr}(\Theta^b)=2\sqrt{-1}\,d\alpha$, hence $4x=s^b$ is the (first) Bismut scalar curvature thus is a global function, and it would be a constant if the universal cover of $M^3$ is homogeneous. Similarly, by considering $(\rho^{b(1)})^2 \wedge \omega$, we see that $x^2-y^2$ is also a constant, hence $y$ must be a constant as well. The all the possibly non-zero Bismut curvature components are $R^b_{i\bar{j}k\bar{\ell}}$ with indices between $1$ and $2$, given by constants $x$, $\pm \sqrt{-1}\,y$, and $\pm (x+2)$.  We claim that for any $1\leq i,j,k, \ell \leq 2$, it always holds that
\begin{equation} \label{eq:5.1}
 \sum_r \big( R^b_{r\bar{j}k\bar{\ell}} \,\theta^b_{ir} - R^b_{i\bar{r}k\bar{\ell}} \,\theta^b_{rj} \big) = 0.
 \end{equation}
To see this, notice that $\theta^b_{11}=\theta^b_{22}$, $\theta^b_{12}=-\theta^b_{21}$,  $\Theta^b_{11}=\Theta^b_{22}$, $\Theta^b_{12}=-\Theta^b_{21}$, and $R^b_{i\bar{j}k\bar{\ell}} = R^b_{k\bar{\ell}i\bar{j} }$. By letting $ij=11$, $22$, or $12$, we verify that (\ref{eq:5.1}) holds. Similarly, $ \sum_r \big( R^b_{i\bar{j}r\bar{\ell}} \,\theta^b_{kr} - R^b_{i\bar{j}k\bar{r}} \,\theta^b_{r\ell} \big) = 0$. Therefore,
$$ (\nabla^bR^b)_{i\bar{j}k\bar{\ell}} = d (R^b_{i\bar{j}k\bar{\ell}} ) - \sum_r \big( R^b_{r\bar{j}k\bar{\ell}} \,\theta^b_{ir} - R^b_{i\bar{r}k\bar{\ell}} \,\theta^b_{rj} + R^b_{i\bar{j}r\bar{\ell}} \,\theta^b_{kr} - R^b_{i\bar{j}k\bar{r}} \,\theta^b_{r\ell} \big) = d (R^b_{i\bar{j}k\bar{\ell}} ) = 0 $$
as $x$ and $y$ are constants. This shows that $\nabla^bR^b=0$, hence $g$ is {\rm BAS}, and we have completed the proof of Proposition \ref{prop5.3}.
\end{proof}
\begin{remark} It is worth mentioning that compact homogeneous $3$-folds have been classified in \cite{T}. Among these we find the Wallach $3$-fold $\SU(3)/\T$ which was already studied in Theorem \ref{SU}, and also the Calabi Eckmann manifold $N:= \SU(2)\times\SU(2)$. The manifold $N$ is endowed with a Samelson structure, but the Hermitian metric is supposed to be simply left-invariant (cf. Theorem \ref{Sam}).  \end{remark}

\begin{remark}
From the proof above, we see that the answer to Question \ref{question5.1} would be positive for $n=3$ if and only if any compact locally homogeneous non-K\"ahler Vaisman threefold $(M^3,g)$ are {\rm BAS}. Such a manifold is a metric fiber bundle over $S^1$ with fiber being a compact, locally homogeneous Sasakian $5$-manifold $N^5$.
\end{remark}
 Our last result deals with nilpotent Lie groups endowed with left-invariant nilpotent complex structures.
\begin{proof}[{\bf Proof of Proposition \ref{prop5.5}}]
Let $(M^n,g)$ be a compact non-K\"ahler Hermitian manifold whose universal cover is $(G,J,\tilde{g})$ where $G$ is a nilpotent Lie group, $\tilde{g}$ (which is the lift of $g$) is  a left-invariant metric,  and $J$ a left-invariant complex structure on $G$ compatible with $g$. We also assume that $J$ is nilpotent in the sense of Cordero, Fern\`andez, Gray, and Ugarte \cite{CFGU}. If $g$ is {\rm BTP}, then by \cite[Proposition 4.4]{ZhaoZ22}, we know that there exists a left-invariant unitary frame $e$ with dual coframe $\varphi$ on $G$ and a positive integer $r<n$ such that
\begin{equation}  \label{eq:5.2}
\left\{ \begin{array} {ll}
  d\varphi_i \,  = 0, & 1\leq i\leq r; \\
  d\varphi_{\alpha} = \sum_{i=1}^r Y_{\alpha i} \varphi_i \wedge \overline{\varphi}_i, & r+1\leq \alpha \leq n.
 \end{array}  \right.
\end{equation}
Here those $Y_{\alpha i}$ are constants, and $r<n$ as $g$ is not K\"ahler. In particular, $G$ is $2$-step nilpotent, and the complex structure $J$ is abelian, meaning that $[Jx,Jy] =[x,y]$ for any $x$, $y$ in the Lie algebra ${\mathfrak g}$ of $G$. Following \cite{VYZ}, denote the structure constants by
$$ C^j_{ik} = \langle [e_i,e_k],\overline{e}_j\rangle, \ \ \ D^j_{ik} = \langle [\overline{e}_j,e_k],e_i\rangle . $$
Then we have
\begin{equation*}
 T^j_{ik} = - C^j_{ik} - D^j_{ik} + D^j_{ki}\,, \ \ \ \ \
  \theta^b_{ij} = \sum_k \big( (C^j_{ki}+ D^j_{ki}) \varphi_k - \overline{(C^i_{kj}+ D^i_{kj})} \,\overline{ \varphi}_k  \big) \,.
 \end{equation*}
and the structure equation is given by
\begin{equation} \label{eq:5.3}
 d\varphi_j = -\frac{1}{2} \sum_{i,k} C^j_{ik} \varphi_i\wedge \varphi_k - \sum_{i,k} \overline{D^i_{jk}} \varphi_i \wedge \overline{\varphi}_k.
 \end{equation}
 Also, by \cite[Lemma 1]{LZ}, the Chern curvature $R$ is given by
\begin{equation} \label{eq:5.4}
R_{i\bar{j}k\bar{\ell}} = \sum_{s} \big( D^s_{ki} \overline{ D^s_{\ell j} } -   D^{\ell}_{si} \overline{ D^k_{s j} } -  D^{j}_{si} \overline{ D^k_{ \ell s} } -   \overline{ D^i_{s j} } D^{\ell}_{ks}  \big)
\end{equation}
Comparing equations (\ref{eq:5.2}) with (\ref{eq:5.3}), we see that for our $G$ we have $C=0$ and the only possibly non-zero $D$ components are
$$ D^i_{\alpha i}=-\overline{Y_{\alpha i} } , \ \ \ 1\leq i\leq r, \ r+1\leq \alpha \leq n. $$
In particular, the only possibly non-zero entries of $\theta^b$ under $e$ are
\begin{equation} \label{eq:5.5}
\theta^b_{ii} = \sum_{\alpha =r+1}^n \big( Y_{\alpha i} \overline{\varphi}_{\alpha} - \overline{Y_{\alpha i}} \varphi_{\alpha}  \big), \ \ \ \ 1\leq i\leq r.
\end{equation}
Also, the last two terms on the right hand side of (\ref{eq:5.4}) vanish as $D^*_{s*}=0$ for $s\leq r$ and $D^*_{*s}=0$ for $s>r$. Thus the only possibly non-zero components of $R$ are
\begin{equation}
R_{i\bar{i}\alpha \bar{\beta}}  = \overline{Y_{\alpha i}} Y_{\beta i}, \ \ \ R_{i\bar{j}j\bar{i}} = -\sum_{\alpha =r+1}^n \overline{Y_{\alpha i}} Y_{\alpha j},   \ \ \ \forall \ 1\leq i,j\leq r, \, \forall \ r+1\leq \alpha ,\beta \leq n.
\end{equation}
 As we already noticed in \eqref{curv}, for {\rm BTP} manifolds we have that $\nabla^bR^b=0$ if and only if $\nabla^bR=0$. For any $1\leq a,b,c,d\leq n$, since all $R_{a\bar{b}c\bar{d}}$ are constants and $\theta^b$ is diagonal, we have
\begin{eqnarray*}
(\nabla^b R)_{a\bar{b}c\bar{d}}  & = & d(R_{a\bar{b}c\bar{d}}) - \sum_s \big( R_{s\bar{b}c\bar{d}} \theta^b_{as} - R_{a\bar{s}c\bar{d}} \theta^b_{sb} + R_{a\bar{b}s\bar{d}} \theta^b_{cs} - R_{a\bar{b}c\bar{s}} \theta^b_{sd} \big) \\
& = &  - R_{a\bar{b}c\bar{d}} \,\big( \theta^b_{aa} -  \theta^b_{bb} + \theta^b_{cc} - \theta^b_{dd} \big)
\end{eqnarray*}
In the last line, the first factor $R_{a\bar{b}c\bar{d}}$ is non-trivial only if $(abcd)=(ii\alpha\beta)$ or $(ijji)$, but in both of these cases the second factor vanishes by (\ref{eq:5.5}). Thus $\nabla^bR=0$, hence $\nabla^bR^b=0$ and $g$ is {\rm BAS}. This completes the proof of Proposition \ref{prop5.5}.
\end{proof}


\begin{thebibliography}{99}
\bibitem{Ag} I. Agricola,
\newblock\textit{The Snri Lectures on non-integrable geometries with torsion}
\newblock Arch. Math. (Brno) \textbf{42} (2006), 5-84.

\bibitem{A} D.N. Akhiezer,
\newblock \textit{Lie Group actions in complex Analysis, }
\newblock  Aspects of Math., Vieweg \textbf{27} (1995)

\bibitem{AD} D.V. Alekseevsky and L. David,
\newblock \textit{ A note about invariant SKT structures and generalized K\"ahler structures on flag manifolds,}
\newblock Proc. Edinburgh Math. Soc. \textbf{55} (2012), 543-549.

\bibitem {AS} W. Ambrose and I.M. Singer,
\newblock {\em On homogeneous Riemannian manifolds,}
\newblock Duke Math. J., {\bf 25} (1958), 647-669.

\bibitem{AV} A. Andrada and R. Villacampa,
\newblock \textit{Bismut connection on Vaisman manifolds,}
\newblock Math.Z., {\bf 302} (2022), 1091-1126.


\bibitem{AOUV} D. Angella, A. Otal, L. Ugarte, R. Villacampa,
\newblock \textit{On Gauduchon connections with K\"ahler-like curvature},
\newblock  Commun. Anal. Geom. {\bf 30} (2022), no.5, 961-1006.


\bibitem{AC} A. Arvanitoyeorgos, I. Chrysikos,
\newblock \textit{Motion of charged particles and homogeneous geodesics in K\"ahler c-spaces with two isotropy summands,}
\newblock Tokyo J. Math. {\bf 32} (2009), 487-500.

\bibitem{B00} F. Belgun,
\newblock \textit{On the metric structure of non-K\"ahler complex surfaces},
\newblock Math. Ann. \textbf{317} (2000), 1-40.


\bibitem{B12} F. Belgun,
\newblock \textit{On the metric structure of some non-K\"ahler complex threefolds,}
\newblock arXiv: 1208.4021.

\bibitem {Bismut} J.-M. Bismut,
\newblock\textit{A local index theorem for non-K\"ahler manifolds,}  Math. Ann. {\bf 284} (1989), no. 4, 681-699.

\bibitem{Boothby} W. Boothby,
\newblock \textit{Hermitian manifolds with zero curvature},
\newblock Michigan Math. J. {\bf 5}, (1958), no.2, 229-233.


\bibitem {BFR} M. Bordermann, M. Forger and H. R\"omer,
\newblock \textit{\em Homogeneous K\"ahler Manifolds: paving the way towards new supersymmetric Sigma Models,}
\newblock Comm. Math. Phys. \textbf{ 102} (1986), 605-647

\bibitem{CFGU} L. Cordero, M. Fern\'{a}ndez, A. Gray, L. Ugarte,
\newblock \textit{Compact nilmanifolds with nilpotent complex structures: Dolbeault cohomology},
\newblock Trans. Amer. Math. Soc., \textbf{352} (2000), no.12, 5405-5433.


\bibitem{CS} R. Cleyton, A. Swann,
\newblock\textit{Einstein metrics via intrinsic or parallel torsion}
\newblock Math. Z. \textbf{247} (2004), 513-528

\bibitem{CMS} R. Cleyton, A. Moroianu, U. Semmelmann,
\newblock{Metric connections with parallel skew-symmetric torsion}
\newblock Adv. Math. \textbf{378} (2021), art. no. 107519

\bibitem{FG} A. Fino, G. Grantcharov,
\newblock{CYT and SKT metrics on compact semi-simple Lie groups}
\newblock Sigma texbf{19} (2023), paper no. 28, Special issue in honour of Jean-Pierre Bourguignon

\bibitem{FT} A. Fino, N. Tardini,
\newblock \textit{Some remarks on Hermitian manifolds satisfying K\"ahler-like conditions},
\newblock Math. Zeit. \textbf{298} (2021), 49-68.


\bibitem{Gaud} P. Gauduchon,
\newblock\textit{Hermitian connections and Dirac operators}
\newblock Boll. Un. Mat. Ital. B (7) \textbf{11} (1997), 257-288.

\bibitem{G} C.S. Gordon,
\newblock\textit{Naturally reductive homogeneous Riemannian manifolds}
\newblock Can. J. Math. \textbf{37} (1985), 467-487.


\bibitem{He} S. Helgason,
\newblock\textit{Differential Geometry, Lie groups and symmetric spaces}
\newblock Academic Press, Inc. (1978).



\bibitem{KN} S.~Kobayashi, K.~Nomizu.
\newblock {\em Foundations of differential geometry. Vol.~{II}}.
\newblock Interscience Publishers, New York-London, 1969.



\bibitem{LPV} R.A. Lafuente, M. Pujia, L. Vezzoni,
\newblock\textit{Hermitian curvature flow on unimodular lie groups and static invariant metrics,}
\newblock Trans. Amer. Math. Soc., \textbf{373} (2020), 3967-3993.

\bibitem{LZ} Y. Li, F. Zheng,
\newblock \textit{Complex nilmanifolds with constant holomorphic sectional curvature},
\newblock Proc. Amer. Math. Soc., \textbf{150} (2022), 319-326.


\bibitem{NiZheng} L. Ni, F. Zheng,
\newblock \textit{       A classification of locally Chern homogenous Hermitian manifolds, }
\newblock arXiv: 2301.00579

\bibitem{OV} L. Ornea, M. Verbitsky,
\newblock \textit{Principle of Locally Conformally K\"ahler Geometry},
\newblock arXiv: 2208:07188  (2023).



\bibitem{Samelson} H. Samelson,
\newblock \textit{A class of complex analytic manifolds },
\newblock Portugal. Math., \textbf{12} (1953), 129–132.


\bibitem{Sekigawa} K. Sekigawa,
\newblock {\em Notes on homogeneous almost Hermitian manifolds,}
\newblock Hokkaido Math. J., {\bf 7} (1978), 206-213.


\bibitem {ST} J. Streets and G. Tian, \emph{Regularity results for pluriclosed flow,} Geometry and Topology, {\bf 17} (2013), 2389-2429.

\bibitem {Strominger} A. Strominger,
\newblock \emph{Superstrings with Torsion,}
\newblock Nuclear Phys. B {\bf 274} (1986), 253-284.

\bibitem{T} J. Tits,
\newblock \textit{Espaces homog\`enes complexes compacts,}
\newblock Comment. Math. Helv., \textbf{37} (1962), 111-120.

\bibitem{TV} F.Tricerri, L. Vanhecke,
\newblock \emph{Homogeneous Structures on Riemannian Manifolds,}
\newblock L.M.S. Lecture Note Series, Cambridge Univ. Press {\bf 83} (1983)

\bibitem{Tsukada} K. Tsukada,
\newblock\emph{Eigenvalues of the Laplacian on Calabi-Eckmann manifolds,}
\newblock J. Math. Soc. Japan, \textbf{33} (1981), 673-691.



\bibitem{WYZ} Q. Wang, B. Yang, F. Zheng,
\newblock \textit{On Bismut flat manifolds,}
\newblock Trans. Amer. Math. Soc. \textbf{373} (2020), 5747-5772.

\bibitem {VYZ} L. Vezzoni, B. Yang, F. Zheng,
\newblock \textit{Lie groups with flat Gauduchon connections},
\newblock Math. Zeit. \textbf{293} (2019), 597-608.



\bibitem{YZZ} S.-T. Yau, Q. Zhao, F. Zheng,
\newblock \textit{On Strominger K\"ahler-like manifolds with degenerate torsion},
\newblock Trans. Amer. Math. Soc. {\bf 376} (2023), no.5, 3063-3085.


\bibitem{ZhaoZ19Str} Q. Zhao, F. Zheng,
\newblock \textit{Strominger connection and pluriclosed metrics},
\newblock J. Reine Angew. Math. (Crelles) { \bf 796} (2023), 245-267.

\bibitem{ZhaoZ19Nil} Q. Zhao, F. Zheng,
\newblock \textit{Complex nilmanifolds and K\"ahler-like connections},
\newblock J. Geom. Phys. {\bf 146} (2019).

\bibitem{ZhaoZ22} Q. Zhao, F. Zheng,
\newblock \textit{On Hermitian manifolds with Bismut-Strominger parallel torsion},
\newblock arXiv:2208.03071







\end{thebibliography}
\end{document}